\documentclass{article}
\usepackage{amsmath}    
\usepackage{longtable}  
\usepackage{latexsym}	
\usepackage{amssymb}	
\usepackage{epsfig}
\usepackage{algorithmic}

\newtheorem{definition}{Definition}[section]
\newtheorem{algorithm}{Algorithm}[section]

\newtheorem{theorem}{Theorem}[section]
\newtheorem{lemma}{Lemma}[section]

\newtheorem{corollary}{Corollary}[section]
\newtheorem{remark}{Remark}[section]

\newcommand{\qed}{\nobreak \ifvmode \relax \else \ifdim\lastskip<1.5em \hskip-\lastskip \hskip1.5em plus0em minus0.5em \fi \nobreak \vrule height0.75em width0.5em depth0.25em\fi} 

\newenvironment{proof}[1][Proof:]{\begin{trivlist} 
\item[\hskip \labelsep {\bfseries #1}]}{\end{trivlist}} 

\def\0{\bf \0}

\def\B{{\bf B}}

\def\E{{\bf E}}

\def\H{{\bf H}}
\def\I{{\bf I}}

\def\0{{\bf 0}}

\def\R{\mathbb{R}}

\def\T{{\bf T}}

\def\d{{\bf d}}

\def\h{{\bf h}}

\def\p{{\bf p}}

\def\s{{\bf s}}

\def\v{{\bf v}}

\def\x{{\bf x}}
\def\y{{\bf y}}
\def\z{{\bf z}}

\def\T{{\rm T}}

\begin{document}

\title{An arc-search BFGS algorithm for unconstrained nonlinear optimization problems}

\author{
Yaguang Yang\thanks{Goddard Space Flight Center, NASA, 
8800 Greenbelt Rd, Greenbelt, MD 20771. 
Email: yaguang.yang@verizon.net}}

\date{\today}

\maketitle    

\begin{abstract}
The classical BFGS algorithm performs excellently for convex optimization problems. However, for non-convex problems, the classical BFGS method may fail to converge reliably. To overcome this limitation, researchers have developed modified BFGS methods that are applicable to both convex and non-convex optimization problems. Among these methods, a robust BFGS algorithm has been shown to achieve global convergence and fast local convergence, with a superlinear convergence rate, for both convex and non-convex nonlinear optimization problems under mild assumptions. In this paper, we propose an arc-search BFGS algorithm that aims to further improve the computational efficiency of the robust BFGS method while preserving its desirable convergence properties. Numerical experiments are carried out, and performance comparisons between the proposed algorithm and state-of-the-art algorithms are reported to demonstrate the advantages of the arc-search BFGS algorithm.
\end{abstract}

{\bf keywords:} arc-search BFGS algorithm, global convergence, superlinear convergence,
unconstrained optimization.


\section{ Introduction}
BFGS algorithm, introduced in the seminal works of C. G. Broyden, R. Fletcher, D. Goldfarb, and D. F. Shanno \cite{broyden70,fletcher70,goldfarb70,shanno70}, is one of the most successful algorithms for unconstrained optimization problems. The algorithm is globally convergent when the objective function is convex \cite{nocedal99}. However, for nonconvex objective functions, Y.-H. Dai showed in \cite{dai02} that the classical BFGS method may fail to converge. Consequently, many modified BFGS methods have been proposed; see, for example, \cite{li01,llw26,yang24,yqh23,yzlc24,zhang05} and the references therein. Among these approaches, the robust BFGS method proposed in \cite{yang24} (based on an idea of \cite{yang15}) has been shown to possess desirable theoretical properties, including global convergence and a local superlinear convergence rate. Moreover, preliminary numerical experiments reported in \cite{yang24} demonstrate that it outperforms several state-of-the-art algorithms.

The goal of this paper is to further improve the efficiency and effectiveness of the robust BFGS method while preserving its favorable convergence properties. To this end, we view the unconstrained optimization problem as a special case of a constrained optimization problem with box constraints, where each variable satisfies $x_i \in (-\infty,\infty),~i=1,2,…,n$. Under this perspective, unconstrained optimization can be treated within a unified constrained optimization framework, which provides additional flexibility for designing modified quasi-Newton updating strategies.

Within this framework, every iterate is an interior point, which naturally connects unconstrained optimization algorithms to the Interior-point method (IPM). Developed in the 1980s, widely popularized in the 1990s, and now recognized as one of the most efficient approaches for constrained optimization \cite{ly84,nocedal99,yang20}, IPMs have been extensively studied in the literature. More recently, it has been shown that arc-search interior-point methods possess superior convergence properties compared with traditional line-search interior-point methods for linear optimization \cite{iy24,yang18}, convex quadratic optimization \cite{yang11,zhl22}, symmetric optimization \cite{kheirfam16,pzm18,ylz17}, linear complementarity problems \cite{sd20,yzh18}, semidefinite optimization \cite{km18,kok21,zyzlh19}, and nonlinear optimization \cite{yiy22,yang25}. In addition, the author showed in \cite{yang18} that an arc-search interior-point method not only achieves the best known polynomial complexity bound for linear optimization problems, but is also competitive with Mehrotra’s algorithm, which has been implemented in many state-of-the-art optimization software packages \cite{wright97}. This resolves a long-standing dilemma discussed in \cite{todd02,wright97,yang20}, namely, how to design an interior-point method based on arc-search techniques that simultaneously achieves strong theoretical polynomial complexity and excellent practical computational performance.

Motivated by the success of the arc-search in IPM, we propose an arc-search BFGS algorithm that combines the advantages of the robust BFGS framework with those of arc-search techniques. We show that the proposed arc-search algorithm naturally satisfies the Wolfe conditions, thereby eliminating the need for a complicated procedure described in \cite{mt90}, which is typically essential for ensuring convergence~\cite{nocedal99}. Furthermore, we establish the global convergence and superlinear convergence rate of the proposed method. Numerical experiments on standard CUTEst benchmark set problems demonstrate that the proposed arc-search BFGS algorithm is highly promising in practice.

Throughout the paper, we use lowercase bold font for vectors, uppercase bold font for matrices, standard font for scalars. To save the space, we use $\x=(x_1, x_2, \ldots, x_n)$ for column vector $\x$. Therefore, we define by $\nabla_{\x} f(\x)=(\frac{df}{dx_1}, \frac{df}{dx_2}, \ldots \frac{df}{dx_n})$ or simply $\nabla_{\x} f$ for the gradient of $f(\x)$, by $\H(\x)=\nabla_{\x}^2 f(\x)$ or simply $\H$ for the Hessian of $f(\x)$. We denote $\H \succ \0$ if a matrix $\H$ is positive definite, $\H \succeq \0$ if $\H$ is positive semi-definite. We  use superscript $k$ to represent the value of a vector (or a matrix) at the $k$th iteration while use the subscript $i$ ($ij$) to represent the $i$ element ($ij$ element) of the vector (or the matrix), and subscript $k$ for the value of the scalar at the $k$ iteration while use the superscript for the power of the scalar. Hence, $\x^0$ is used to represent the initial point of the algorithm and $\alpha_0$ is the initial value of $\alpha$. 

The remainder of the paper is organized as follows. Section~\ref{secarcBFGS} introduces the arc-search BFGS algorithm. Section~\ref{convergenceAnalysis} discusses the algorithm's convergence properties. Section~\ref{implementationDetails} provides the detailed implementation information.  Section~\ref{numericalTest} includes the test results and comparison of proposed algorithm with the established ones. Section~\ref{conclusions} summarizes the conclusions.

\section{The arc-search BFGS Method}\label{secarcBFGS}

This paper proposes an arc-search algorithm to minimize a general (convex or nonconvex) multi-variable nonlinear function
\begin{equation}
\min f(\x),
\label{obj}
\end{equation}
where $f: \R^n \rightarrow \R$ is a twice differentiable function of $\x \in \R^{n}$. Let $\nabla_{\x} f({\x})$ and $\H(\x)$ be the gradient and Hessian matrix of $f(\x)$ respectively. We denote by $\bar{\x}$ a local minimizer of (\ref{obj}) that satisfies
\begin{equation}
\nabla_{\x} f(\bar{\x})=\0.
\label{necc}
\end{equation}

Following the practice of \cite{nocedal99}, we make the standard assumptions on $f(\x)$ which will be used in the convergence analysis. 

\medskip
{\bf Assumptions:}
\begin{itemize}
\item[1.] For an open set ${\cal M}$ containing the level set 
${\cal L}=\{ \x: f(\x) \le f(\x^0) \}$, $\nabla f(\x)$ is Lipschitz continuous, i.e., there exists a constant $L>0$ such that
\begin{equation}
\| \nabla f(\x)-\nabla f(\y)\| \le L \| \x-\y\|,
\label{function}
\end{equation}
for all $\x, \y \in {\cal M}$.
\item[2.] There are positive numbers $\delta>0$, $0<m <1 <M$, and a neighborhood of $\bar{\x}$, defined by 
${\cal N}(\bar{\x})=\{\x: f({\x})-f(\bar{\x}) \le \delta \}$, such that for all
$\x \in {\cal N}(\bar{\x})$ and for all $\z \in \R^n$,
\begin{equation}
m \| \z\|^2 \le \z^{\T} \H(\x) \z \le M \| \z\|^2.
\label{convex}
\end{equation}
\item[3.] There is a positive number $L >0$ such that for all $\x \in {\cal N}(\bar{\x})$,
\begin{equation}
\| \H(\x) - \H(\bar{\x}) \| \le L\| \x-\bar{\x}\|.
\label{Matrix}
\end{equation}
\item[4.] The summation of the distances between the iterates and optimizer is bounded, i.e.,
\begin{equation}
\sum_{k=1}^{\infty} \| \x^k - \bar{\x} \| < \infty.
\label{limitK}
\end{equation}
\end{itemize}
Assumption 1 \cite[(3.13)]{nocedal99} is required when applying the Zoutendijk theorem to establish the global\footnote{We say an algorithm is globally convergent if the convergence does not depend on the selection of the initial point.} convergence for the arc-search BFGS algorithm. Assumption 2 is required to establish the superlinear convergence. It is proved in \cite[Assumption 8.1]{nocedal99} that for all $\x \in {\cal N}(\bar{\x})$, if this strong second order sufficient condition holds, then, there is a unique minimum in the neighborhood ${\cal N}(\bar{\x})$. In Lemma \ref{conditionM}, we establish a relationship between this assumption and the arc-search BFGS algorithm (see Eq. (\ref{gamma}) and the discussion of the selection of $\gamma_k \in [0,1]$). Assumption 3 \cite[Assumption 8.2]{nocedal99} is also needed in the proof of the superlinear convergence. Assumption 4 coincides with \cite[(8.52)]{nocedal99} and is used to prove the superlinear convergence rate. The parameter $L$ in (\ref{function}) may differ from the parameter $L$ in (\ref{Matrix}). However, we can always choose the larger one so that both (\ref{function}) and (\ref{Matrix}) hold for the same value of L, thereby simplifying the notation and the proofs.

\begin{remark}
The first three assumptions are related to the optimization problem itself, while the last assumption pertains to the specific algorithm. It is worthwhile to note that any sufficiently smooth function with isolated local minimizers belongs to the class of functions satisfying the first three assumptions.
\end{remark}

To find a numerical solution of the nonlinear problem (\ref{obj}) using classical optimization methods, the following iteration is carried out until an optimal solution is found
\begin{equation}
	\x^{k+1} = \x^k + \alpha_k \d^k,
\end{equation}
where $\x^k$ is current iterate, $\x^{k+1}$ is the next iterate, $\alpha_k$ is the step size which is normally determined by trial and error method, such as Armijo rule, and $\d^k$ is the search direction. In general, a search direction should be a descent direction as defined below.

\begin{definition}
	A search direction $\d^k$ is descent if $\nabla f(\x^k)^{\T} \d <0$. 
\end{definition}

Denote $\H^k=\H(\x^k)$. A ultra-fast algorithm is the Newton's method whose search direction is determined by
\begin{equation}
	\H^k \d^k = - \nabla f(\x^k).
	\label{newtonD}
\end{equation}
Although Newton's method converges fast locally, the computation of $\H(\x^k)$ is demanding, and the method oftentimes does not converge if an initial point is not close to an optimizer. To avoid high computational cost of the Hessian matrix $\H(\x^k)$, A popular strategy is to optimize a quadratic system 
\begin{equation}
	m_k(\d) = f(\x^k) + \nabla f(\x^k)^{\T} \d + \frac{1}{2} \d^{\T} \B^k \d,
	\label{quadraticM}
\end{equation}
in every iteration $k$, where $\B^k$ is an approximation of $\H^k$, and the search direction $\d^k$ can be determined by the linear system of equations
\begin{equation}
	\B^k \d^k = - \nabla f(\x^k).
	\label{bfgsD}
\end{equation}
The positive definite matrix $\B^k$ can be updated by the DFP formula as follows \cite{nocedal99}:
\begin{equation}
 	\B^{k+1}=\B^k-\frac{\B^k \s^k \s^{k^{\T}} \B^k}{\s^{k^{\T}}\B^k \s^k}
 	+\frac{\y^k \y^{k^{\T}}}{\y^{k^{\T}} \s^k}.
 	\label{BFGS}
\end{equation}
where
\begin{equation}
	\s^k=\x^{k+1}-\x^k,
	\label{sk}
\end{equation}
\begin{equation}
	\y^k=\nabla f(\x^{k+1}) - \nabla f(\x^{k}).
	\label{yk}
\end{equation}
Unlike the computation of $\H(\x^k)$ which involves the second-order derivatives, the computation of $\B^k$ involves only $\x^{k-1}$, $\x^{k}$, and the first-order derivatives $\nabla f(\x^{k-1})$ and $\nabla f(\x^{k})$, therefore is much faster. By multiplying $\s^k$ from the right on both sides of (\ref{BFGS}), we have 
\begin{equation}
	\B^{k+1} \s^k=\y^k.
	\label{secant1}
\end{equation}
The DFP (and its variant BFGS) algorithm has been very popular, but it may not be convergent for nonconvex problem \cite{dai02}. Since the steepest descent algorithm is known to be globally convergent, a robust BFGS, which is based on an idea of \cite{yang15}, is developed in \cite{yang24} that combines the merits of the steepest descent algorithm and BFGS algorithm by considering the quadratic system
\begin{equation}
	n_k(\d) = f(\x^k) + \nabla f(\x^k)^{\T} \d + \frac{1}{2} \d^{\T} (\gamma_k \I + (1-\gamma_k)\B^{k} ) \d = f(\x^k) + \nabla f(\x^k)^{\T} \d + \frac{1}{2} \d^{\T} \E^{k} \d,
	\label{robustM}
\end{equation}
where $\gamma_k \in [0,1]$ is dynamically updated in every iteration. It will be proven that $\gamma_k = 0$ when iterates are in a neighborhood of a minimizer where the function $f$ is locally convex. Therefore, the robust BFGS search direction is defined by
\begin{equation}
	\E^{k} \d^k=- \nabla f(\x^{k}), 
	\label{dirMBFGS}
\end{equation}
and 
the search direction $\d_k$ is calculated by 
\begin{equation}
	\d^k=-(\E^{k})^{-1}  \nabla f(\x^{k}), 
	\label{calMBFGS}
\end{equation}
where the superscript $k$ inside the bracket represents the number of iteration and the superscrit $-1$ outside the bracket represents the inverse of the matrix $(\E^{k})$. It has been shown in \cite{yang24} that $\E^{k+1}$ can be obtained iteratively by
\begin{equation}
	\E^{k+1}=\E^k-\frac{\E^k \s^k \s^{k^{\T}} \E^k}{\s^{k^{\T}}\E^k \s^k}
	+\frac{\z^k \z^{k^{\T}}}{\z^{k^{\T}} \s^k},
	\label{arcBFGS}
\end{equation}
where
\begin{equation}
	\z^k = \gamma_k \s^k +(1-\gamma_k)\y^k.
	\label{zk}
\end{equation}
To avoid the computation of the inverse matrix of $(\E^{k})^{-1}$, using the Sherman-Morrison-Woodbury formula \cite[page 51]{gv89} for (\ref{arcBFGS}), the robust BFGS update is give by
\begin{equation}
	\left( \E^{k+1} \right)^{-1}=
	\left(\I-\frac{ \s^{k} \z^{{k}^{\T}}}{\z^{{k}^{\T}} \s^{k}}\right) \left( \E^{k}  \right)^{-1}
	\left(\I-\frac{\z^{k} \s^{k^{\T}}}{\z^{{k}^{\T}} \s^{k}}\right)
	+\frac{\s^{k} \s^{{k}^{\T}}}{\z^{{k}^{\T}} \s^{k}}.
	\label{E1}
\end{equation} 
In view of (\ref{secant1}), equation (\ref{zk}) can be written as
\begin{equation}
	\z^k = \gamma_k \s^k +(1-\gamma_k)\y^k=(\gamma_k \I + (1-\gamma_k)\B^{k+1} )\s^k := \E^{k+1} \s^k,
	\label{secant2}
\end{equation}
where 
\begin{equation}
	\E^{k+1}=\gamma_k \I + (1-\gamma_k)\B^{k+1}.
	\label{Ekp1}
\end{equation}

\begin{remark}
	It has been shown in \cite{yang15,yang24} that if $\gamma_k =1$, the search direction $\d^k$ in (\ref{secant2}) is reduced to the steepest descent direction; if $\gamma_k =0$, the search direction $\d^k$ is reduced to the BFGS direction. By a smart selection of $\gamma_k$, the robust BFGS algorithm is globally convergent with a locally superlinear convergence rate \cite{yang24}. To achieve the latter, the robust BFGS algorithm should select $\gamma_k \geq 0$ as close to zero as possible when the iterate approaches to the optimizer. On the other hand, if the iterate is far away from the optimizer and $\B^{k+1}$ is ill-conditioned, according to (\ref{Ekp1}), $\E^{k+1}$ is ill-conditioned if $\gamma_k =0$. As a result of (\ref{dirMBFGS}), it may generate an inaccurate search direction $\d^k$. Therefore, it is important to carefully select $\gamma_k$.
	\label{gamma01}
\end{remark}

The arc-search BFGS is aimed at further improving efficiency of robust BFGS while keeping the nice convergence properties of robust BFGS. It is based on the search direction $\d^k$ of the robust BFGS \cite{yang24} and arc-search techniques developed in \cite{yang20}. The proposed arc-search BFGS update is carried out as follows:
\begin{equation}
	\x^{k+1}=\x^k+\sin(\alpha_k) \d^k -(1-\cos(\alpha_k)) \dot{\d}^k,
	\label{updateX}
\end{equation}
where $\dot{\d}$ is the derivative of $\d$ to be defined in (\ref{ddd}). Rewriting (\ref{updateX}) yields
\begin{equation}
	\frac{\x^{k+1}-\x^k}{\alpha_k}=\frac{\sin(\alpha_k)}{\alpha_k} \d^k -\frac{(1-\cos(\alpha_k))}{\alpha_k} \dot{\d}^k,
	\label{fracX}
\end{equation}
Since $\frac{1}{2}\sin^2(\alpha) \le (1-\cos(\alpha)) \le \sin^2(\alpha)$, let $\alpha_k \rightarrow 0$, it follows from (\ref{fracX}) that
\begin{equation}
	\dot{\x}= \d^k,
	\label{dxdk}
\end{equation}
and
 \begin{equation}
 	\ddot{\x}= \dot{\d}^k \approx \frac{\d^{k}-\d^{k-1}}{\alpha_{k-1}}.
 	\label{ddd}
 \end{equation}
Therefore, the formula (\ref{updateX}) can be written as
\begin{equation}
	\x^{k+1}=\x^k+\sin(\alpha_k) \dot{\x} -(1-\cos(\alpha_k)) \ddot{\x},
	\label{updateX1}
\end{equation}
which is exactly the same as other arc-search algorithms in \cite{yang20} (except the signs of $\dot{\x}$ and $\ddot{\x}$ because here $\d^k=\dot{\x}$ is a descent directon while in \cite{yang20} $-\dot{\x}$ is the descent direction).  Using (\ref{calMBFGS}) and (\ref{updateX}), since $\E^{k}$ is an estimation of $\H^k$ when $\gamma_k=0$, we have another estimation of $\dot{\d}^k$.
\begin{eqnarray}
 	\ddot{\x} & = & \dot{\d}^k =  \frac{ d}{d \alpha_{k-1}} \left( -(\E^{k})^{-1}  \nabla f(\x^{k}) \right)   \nonumber \\
 & = & -(\E^{k})^{-1}  \frac{ d}{d \alpha_{k-1}}  \left( \nabla f(\x^{k-1}+\sin(\alpha_{k-1}){\d}^{k-1}-(1-\cos(\alpha_{k-1})) \dot{\d}^{k-1}) \right)   \nonumber \\
 & = & -(\E^{k})^{-1} \H^{k} \left( \cos(\alpha_{k-1}){\d}^{k-1}-\sin(\alpha_{k-1}) \dot{\d}^{k-1} \right) \nonumber \\
 & \approx & - \cos(\alpha_{k-1}){\d}^{k-1}+\sin(\alpha_{k-1}) \dot{\d}^{k-1}
 	\label{ddd1}
\end{eqnarray}

\begin{remark}
	Both (\ref{ddd}) and (\ref{ddd1}) are implemented. The numerical experiments show that the implementation using (\ref{ddd1}) produces slightly better results, as expected.
	\label{whichFormulate}
\end{remark}

Given $\x^k$, $\d^k$, and $\dot{\d}^k$, denote a map $\h(\alpha_k):  \R \rightarrow \R^n$, where $\h(\alpha_k)=\x^k+\sin(\alpha_k) \d^k -(1-\cos(\alpha_k)) \dot{\d}^k$, it follows
\begin{equation}
 f(\x^{k+1}) = f(\x^k+\sin(\alpha_k) \d^k -(1-\cos(\alpha_k)) \dot{\d}^k)  =f(\h(\alpha_k)) :=F(\alpha_k),
\label{denote1}
\end{equation}
where $F(\alpha_k):  \R \rightarrow \R$. It follows that $F(\alpha_k=0)=f(\x^k)$.

Let $m$ and $M$ be defined as in (\ref{convex}), a critical part of the robust BFGS algorithm (as well as the arc-search BFGS) is to select $\gamma_{k}$ carefully as follows:
\begin{equation}
	\check{\gamma}_k=\begin{cases} \frac{m \s^{k^{\T}} \s^k- \y^{k^{\T}} \s^k}{ \s^{k^{\T}} \s^k- \y^{k^{\T}} \s^k} & \text{if } \s^{k^{\T}} \s^k- \y^{k^{\T}} \s^k \neq 0 \\
0  & \text{otherwise }
\end{cases}
	\label{gammak}
\end{equation}
\begin{equation}
	\begin{array}{ll}
		(\bar{\gamma}_k,\underline{\gamma}_k) & = 
\begin{cases} \frac{( \s^k-\y^k)^{\T}(M \s^k-2 \y^k) \pm \sqrt{(( \s^k- \y^k)^{\T}(M \s^k-2 \y^k))^2-4( \s^k- \y^k)^{\T}( \s^k- \y^k) \y^{k^{\T}}(\y^k-M \s^k)}} 	{2( \s^k- \y^k)^{\T}(\s^k- \y^k)}   & \text{ if } \s^k \neq \y^k \\
0 &  \text{ if } \s^k = \y^k  \end{cases}
\\
& =  \begin{cases} \frac{( \s^k- \y^k)^{\T}(M \s^k-2 \y^k) \pm  \sqrt{(M \s^{k^{\T}}( \s^k- \y^k))^2+4(M-1)( \s^{k^{\T}} \s^k \y^{k^{\T}} \y^k-( \y^{k^{\T}} \s^k)^2)}}		{2( \s^k- \y^k)^{\T}( \s^k- \y^k)} & \text{ if }  \s^k \neq \y^k \\
0 &  \text{ if } \s^k = \y^k  \end{cases},
	\end{array}
	\label{gamma4}
\end{equation}

\begin{equation}
	\begin{cases}
	\gamma_{k} \in [ \max \{ 0, \check{\gamma}_k, \underline{\gamma}_k \}, 1  ]  & 	\text{if }  \s_k^{\T}\s_k > \y_k^{\T}\s_k  \\
	\gamma_{k} \in [ 0,1] & \text{if }  \s_k^{\T}\s_k = \y_k^{\T}\s_k  \\
	\gamma_k \in [\max \{ 0, \underline{\gamma}_k \},  1  ] & \text{if }  \s_k^{\T}\s_k < \y_k^{\T}\s_k.
	\end{cases}
	\label{gamma}
\end{equation}
and 

\begin{remark}
	Although the two formulas in (\ref{gamma4}) are equivalent, the second one is numerically much more stable because it ensures a nonnegative value inside the square root.
\end{remark}

Now we are ready to present the arc-search BFGS algorithm.
\begin{algorithm} {\bf arc-search BFGS} \\
	\begin{algorithmic}[1] 
		\STATE Data:  $0<\epsilon$, $0<m<1<M<\infty$, initial $\x^0$, and $\nabla f(\x^0)$.
		\STATE if $\nabla f(\x^0)=0$, stop.
		\STATE $\d^0 = -\nabla f(\x^0)$, $\dot{\d}^0 = 0$, $\gamma_0=0$, and $\E^0=\I$.
		\FOR{ k=0,1,2,...}
        \STATE Determine $\alpha_k$ satisfying the Wolfe condition (to be discussed later); \label{arcSearch}
        \STATE Set $\x^{k+1}=\x^k+\sin(\alpha_k) \d^k -(1-\cos(\alpha_k)) \dot{\d}^k$; 
		\STATE Calculate gradient $ \nabla f(\x^{k+1})$; 
		\STATE If $\|  \nabla f(\x^{k+1}) \|<\epsilon$, stop; 
		\STATE Compute $\s^{k}=\x^{k+1}-\x^{k}$ and $\y^{k}=\nabla f(\x^{k+1})-\nabla f(\x^{k})$; \label{step8}
		\STATE Select $\gamma_{k}$ satisfying (\ref{gamma}), and set $\gamma_{k}=0$ if $m\s_k^{\T}\s_k \leq \y_k^{\T}\s_k$ and $\y_k^{\T}\y_k \leq M  \y_k^{\T}\s_k$;  \label{step9}
		\STATE Compute $\z^{k}=\gamma_{k}\s^k-(1-\gamma_{k})\y^{k}$;
		\STATE Update $\left(\E^{k+1}\right)^{-1}$ using (\ref{E1}); \label{updateE}
		\STATE Compute search direction $\d^{k+1}=-(\E^{k+1})^{-1}  \nabla f(\x^{k+1})$;
		\STATE If $\gamma_{k}=0$, set $\dot{\d}^{k+1}=\0$; Otherwise $\dot{\d}^{k+1}$ may be approximated using (\ref{ddd}) or (\ref{ddd1}); \label{step12}
		\STATE $k \leftarrow k+1$;
		\ENDFOR
	\end{algorithmic}
	\label{newAlgo}
\end{algorithm}

\begin{remark}
	The computation involving the selection of $\gamma_k$ is negligible (requires ${\cal O}(n)$ operations), and arc-search computational cost in Step \ref{arcSearch} is slightly higher than line search but the additional effort is negligible because the computation of arc-search  (\ref{updateX1}) is simple and cheap. Therefore, the cost of arc-search BFGS in each iteration is almost the same as the cost of the classical BFGS (and line search robust BFGS).
	\label{efficiency}
\end{remark}

\section{Convergence Analysis}\label{convergenceAnalysis}

We first introduce the Wolfe conditions in the context of arc-search. Assume that $\d^k$  is a descent direction, the Wolfe conditions are given as follows.
\begin{subequations}
	\begin{align}
		f(\x^{k+1})=f(\x^k+\sin(\alpha_k) \d^k -(1-\cos(\alpha_k)) \dot{\d}^k) \le f(\x^k)+\sigma_1 \alpha_k \nabla f(\x^k)^{\T} \d^k, \label{C1} \\
		\d^{k^{\T}} \nabla f(\x^{k+1}) = \d^{k^{\T}} \nabla f(\x^k+\sin(\alpha_k) \d^k -(1-\cos(\alpha_k)) \dot{\d}^k) \geq \sigma_2 \nabla f(\x^k)^T \d^k,  \label{C2}
	\end{align}
	\label{WolfeC}
\end{subequations}
where $0 < \sigma_1 < \sigma_2<1$. The existence of the Wolfe conditions is established in \cite{wolfe69,wolfe71}. A very complicated algorithm that finds, in finite steps, a point satisfying the Wolfe conditions is given in \cite{mt90}. We will show that the arc-search BFGS satisfies the Wolfe conditions. But we  first show some nice consequence when Wolfe conditions are applied to the arc-search method (\ref{updateX}).

\begin{lemma}
	Suppose Wolfe conditions hold, $\d^k$ is descent, and $\alpha_k > 0$. then $\y^{k^{\T}} \s^k >0$, which means that both $\s^k \neq \0$ and $\y^k \neq \0$ hold.
	\label{ykskp0}
\end{lemma}
\begin{proof}
First, in view of (\ref{C2}), since $\sigma_2<1$ and $\nabla f(\x^k)^T \d^k <0$, we have
\begin{equation}
	\y^{k^{\T}} {\d^k} =(\nabla f(\x^{k+1})-\nabla f(\x^k))^{\T} {\d}^k  \geq (\sigma_2 -1) \nabla f(\x^k)^T \d^k >0.
\label{middle1}
\end{equation}
The remaining proof is divided into two cases. Case 1 $\y^{k^{\T}} \dot{\d}^k>0$:
If $0<\alpha_k< \frac{\y^{k^{\T}} \d^k}{\y^{k^{\T}}\dot{\d}^k}$, then,
\begin{eqnarray*}
\frac{\y^{k^{\T}} \s^k}{\alpha_k} & = & \y^{k^{\T}} \left( \frac{\d^k \sin(\alpha_k)}{\alpha_k} - \frac{\dot{\d}^k(1-\cos(\alpha_k))}{\alpha_k} \right)  \nonumber \\
& \geq & \frac{\y^{k^{\T}}\d^k \sin(\alpha_k)}{\alpha_k} - \frac{\y^{k^{\T}}\dot{\d}^k \sin^2(\alpha_k)}{\alpha_k}    \nonumber \\
& \geq & \sin(\alpha_k) \left( \frac{\y^{k^{\T}}\d^k}{\alpha_k} - \y^{k^{\T}}\dot{\d}^k \right) >0.
\end{eqnarray*}
Case 2 $\y^{k^{\T}}\dot{\d}^k \leq 0$: In view of (\ref{C2}), (\ref{sk}), and (\ref{updateX}), and (\ref{middle1}), we have
\begin{eqnarray*}
\y^{k^{\T}} \s^k =\y^{k^{\T}}(\d^k \sin(\alpha_k)-\dot{\d}^k(1-\cos(\alpha_k))) \geq \y^{k^{\T}} \d^k \sin(\alpha) \ge (\sigma_2 -1) \nabla f(\x^{k})^{\T} \d^k \sin(\alpha_k) >0 
\end{eqnarray*}
The last inequality holds because $\sigma_2<1$, $\alpha_k>0$, and $\d^k$ is descent.
\hfill \qed
\end{proof}

A corollary follows immediately from Lemma~\ref{ykskp0}.
\begin{corollary}
Suppose Wolfe conditions hold and $\d^k$ is descent. then $\z^{k^{\T}} \s^k >0$, which means that both $\s^k \neq \0$ and $\z^k \neq \0$ hold.
\label{zksk}
\end{corollary}
\begin{proof}
	Multiplying $\s^{k^{\T}}$ to the right on both sides of (\ref{secant2}) and invoking Lemma~\ref{ykskp0} prove the claim.
\hfill \qed
\end{proof}

We will use Taylor's theorem \cite[Theorem 2.1]{nocedal99} in the convergence analysis, which is stated as follows:
\begin{theorem}[Taylor's theorem]
	Suppose that $f: \R^n \rightarrow \R$ is twice continuously differentiable. Then for some $t \in (0,1)$, we have
\begin{equation}
	\nabla_{\x} f(\x+ \p)  = \nabla_{\x} f(\x) + \int_0^1 \nabla_{\x}^2 f(\x+ t\p) \p dt.
\end{equation}
	\label{taylor}
\end{theorem}
For arc-search method, $\p^k=\sin(\alpha_k) \d^k -(1-\cos(\alpha_k))\dot{\d}^k=\x^{k+1}-\x^k=\s^k$, we have the following 
\begin{corollary}
	Suppose $f: \R^n \rightarrow \R$ and $\h(\alpha_k): \R \rightarrow \R^n$ are twice continuously differentiable. Then for some $t \in (0,1)$, we have
	\begin{equation}
		\nabla_{\x}f(\x^{k+1}) -	\nabla_{\x}f(\x^{k}) = \int_0^1 \H(\x^k +t[\sin(\alpha_k) \d^k -(1-\cos(\alpha_k))\dot{\d}^k]) \s^k dt:=\bar{\H} \s^k,
		\label{Hbar}
	\end{equation}
	\label{corollaryTaylor}
\end{corollary}
where the average Hessian matrix $\bar{\H} = \int_0^1 \H(\x^k +t[\sin(\alpha_k) \d^k -(1-\cos(\alpha_k))\dot{\d}^k]) dt$. The following lemmas establishes a relation between Assumption 2 and the arc-search BFGS algorithm.

\begin{lemma}
Let $\gamma_k$ be selected as (\ref{gamma}) and $0<m<1<M<\infty$.  Then, the following inequalities hold.
\begin{equation}
m \le \frac{\z^{k^{\T}} \s^k}{ \s^{k^{\T}} \s^k} \,\,\,\,\, \text{and} \,\,\,\,\,
\frac{ \z^{k^{\T}} \z^k}{ \z^{k^{\T}} \s^k} \le M ,
\label{conditions}
\end{equation}
where $m$ and $M$ are consistent to the ones in Assumption 2. Moreover, if
\begin{equation}
m \le \frac{\y^{k^{\T}} \s^k}{ \s^{k^{\T}} \s^k} \,\,\,\,\, \text{and} \,\,\,\,\,
\frac{ \y^{k^{\T}} \y^k}{ \y^{k^{\T}} \s^k} \le M,
\label{conditionsA}
\end{equation}
then, the formula (\ref{gamma}) is reduced to $\gamma_k \in [0,1]$.
\label{conditionM}
\end{lemma}
\begin{proof}
First, we show if $\gamma_k$ is selected based on the following condition
\begin{equation}
\left\{ 
\begin{array}{ll}
\gamma_k \in [\max \{ \check{\gamma}_k, 0 \}, 1] &
\text{if }  \s_k^{\T}\s_k > \y_k^{\T}\s_k , \\
\gamma_k \in [0, 1],  &
\text{if } \s_k^{\T}\s_k = \y_k^{\T}\s_k, \\
\gamma_k \in [0, 1 ], &
\text{if } \s_k^{\T}\s_k < \y_k^{\T}\s_k
\end{array} \right.
\label{gamma1}
\end{equation}
then the first inequality of (\ref{conditions}) holds. To this end, we write equivalent inequalities of $m \le \frac{\z^{k^{\T}} \s^k}{ \s^{k^{\T}} \s^k}$ as follows:
\begin{eqnarray}
& \z^{k^{\T}}\s^k=(\gamma_k \s^{k^{\T}} +(1-\gamma_k) \y^{k^{\T}} )\s^k 
\ge m \s^{k^{\T}}\s^k \nonumber \\
\iff & \gamma_k(\s^{k^{\T}}\s^k-\y^{k^{\T}}\s^k)\ge m \s^{k^{\T}}\s^k-\y^{k^{\T}}\s^k.
\label{add1}
\end{eqnarray}
Since $m<1$, $\gamma_k \in [0,1]$, and $\s^{k^{\T}}\s^k >0$, it follows
\begin{equation}
\s^{k^{\T}}\s^k-\y^{k^{\T}}\s^k > m \s^{k^{\T}}\s^k-\y^{k^{\T}}\s^k.
\label{add2}
\end{equation}
Let
\begin{equation}
\check{\gamma}_k=\begin{cases} \frac{m \s^{k^{\T}} \s^k- \y^{k^{\T}} \s^k}{ \s^{k^{\T}} \s^k- \y^{k^{\T}} \s^k}, & \text{if } \s^{k^{\T}} \s^k- \y^{k^{\T}} \s^k \neq 0 \\
0, & \text{otherwise.}
\end{cases}
\label{gamma0}
\end{equation}
We divide the discussion into three cases given in (\ref{gamma1}): Case (a) if $\s^{k^{\T}} \s^k- \y^{k^{\T}} \s^k >0$, since $\s^{k^{\T}}\s^k>m \s^{k^{\T}} \s^k$, it follows that $\check{\gamma}_k<1$. Since $\gamma_k \in [0,1]$, it follows that if we select ${\gamma}_k$ satisfying $1 \ge {\gamma}_k \ge \max \{ \check{\gamma}_k, 0 \}$, or equivalently $\gamma_k \in [\max \{ \check{\gamma}_k, 0 \},1]$, then (\ref{add1}) holds. Case (b) if $\s^{k^{\T}} \s^k-\y^{k^{\T}} \s^k =0$, it follows from (\ref{add2}) that $0>m \s^{k^{\T}} \s^k-\y^{k^{\T}} \s^k$, and inequality (\ref{add1}) holds for any $\gamma_k \in [0,1]$. As discussed in Remark~\ref{gamma01}, we prefer to select the smallest $\check{\gamma}_k=0$ when $\x^k$ is close to the minimizer. Case (c) if $\s_k^{\T}\s_k-\y_k^{\T}\s_k <0$, we have $m \s_k^{\T}\s_k-\y_k^{\T}\s_k <\s_k^{\T}\s_k-\y_k^{\T}\s_k <0$ which shows $\check{\gamma}_k>1$. In this case, for any $\gamma_k \in [0,1]$, inequality (\ref{add1}) holds. Clearly, (\ref{add1}) holds means $m \le \frac{\z^{k^{\T}} \s^k}{ \s^{k^{\T}} \s^k}$ holds.

Next, we show if the following condition
\begin{equation}
\gamma_k \in [\max \{ 0, \underline{\gamma}_k \}, \min \{ \bar{\gamma}_k, 1 \} ]
\label{case3}
\end{equation}
holds, then the second inequality of (\ref{conditions}) holds. To this end, we provide two equivalent inequalities of $\frac{ \z^{k^{\T}} \z^k}{ \z^{k^{\T}} \s^k} \le M$ as follows:
\begin{subequations}
\begin{align}
& \z_k^{\T}\z_k=(\gamma_k \s_k +(1-\gamma_k)\y_k )^{\T}(\gamma_k \s_k +(1-\gamma_k)\y_k ) 
\le M (\gamma_k \s_k +(1-\gamma_k)\y_k )^{\T}\s_k \label{a}\\
\iff & p(\gamma_k) \equiv \gamma_k^2(\s_k-\y_k)^{\T}(\s_k-\y_k) 
+ \gamma_k(\s_k-\y_k)^{\T}(2\y_k-M\s_k)+\y_k^{\T}(\y_k-M\s_k) \le 0. \label{b}
\end{align}
\label{Mcondition}
\end{subequations}
If $\s_k \neq \y_k$, $p(\gamma_k)$ is a quadratic and convex function of 
$\gamma_k$. Since $M>1$, it is easy to see that the strict inequality of (\ref{a}) holds for $\gamma_k=1$, hence $p(1)<0$. Therefore $p(\gamma_k)=0$ has two solutions $\underline{\gamma}_k$ and 
$\bar{\gamma}_k$ satisfying $\underline{\gamma}_k<1<\bar{\gamma}_k$ and for any $\gamma_k \in [\underline{\gamma}_k, \bar{\gamma}_k]$, (\ref{a}) holds. From (\ref{b}), if $\s_k \neq \y_k$,
\begin{equation}
\begin{array}{ll}
\underline{\gamma}_k & = \frac{(\s_k-\y_k)^{\T}(M\s_k-2\y_k)-
\sqrt{((\s_k-\y_k)^{\T}(M\s_k-2\y_k))^2-4(\s_k-\y_k)^{\T}(\s_k-\y_k)\y_k^{\T}(\y_k-M\s_k)}}
{2(\s_k-\y_k)^{\T}(\s_k-\y_k)}  \\
& = \frac{(\s_k-\y_k)^{\T}(M\s_k-2\y_k)-
\sqrt{(M\s_k^{\T}(\s_k-\y_k))^2+4(M-1)(\s_k^{\T}\s_k\y_k^{\T}\y_k-(\y_k^{\T}\s_k)^2)}}
{2(\s_k-\y_k)^{\T}(\s_k-\y_k)}.
\end{array}
\label{gamma2}
\end{equation}
If $\s_k = \y_k$, then, inequality (\ref{b}) reduces to $(1-M)\y_k^{\T}\y_k \le 0$ which holds for $\forall {\gamma}_k \in [0,1]$. Therefore, we may set $\underline{\gamma}_k=0$.
Since $\gamma_k \in [0,1]$, we have
\begin{equation}
\gamma_k \in [\max \{ 0, \underline{\gamma}_k \},  1  ].
\label{gamma3}
\end{equation} 
Combining (\ref{gamma1}) and (\ref{gamma3}), we conclude that if
\begin{equation}
	\begin{cases}
		\gamma_{k} \in [ \max \{ 0, \check{\gamma}_k, \underline{\gamma}_k \}, 1  ]  &
		\text{if }  \s_k^{\T}\s_k > \y_k^{\T}\s_k , \\
		\gamma_{k} \in [ 0,1] & \text{if }  \s_k^{\T}\s_k = \y_k^{\T}\s_k , \\
		\gamma_k \in [\max \{ 0, \underline{\gamma}_k \},  1  ] & \text{if }  \s_k^{\T}\s_k < \y_k^{\T}\s_k, 
	\end{cases}
\label{inProof}
\end{equation}
then, both inequalities of (\ref{conditions}) hold.

Finally, if the first inequality of  (\ref{conditionsA}) holds, i.e., $m \s^{k^{\T}}\s^k-\y^{k^{\T}}\s^k \leq 0$, then for any $\gamma_k \in [0 ,1]$ the inequality of (\ref{add1}) holds, this means $\check{\gamma}_k =0$. If the second inequality of  (\ref{conditionsA}) holds, i.e., $\y_k^{\T}(\y_k-M\s_k) \le 0$, then $p(0) \le 0$, this means $\underline{\gamma}_k \leq 0$. Substituting $\check{\gamma}_k =0$ and $\underline{\gamma}_k \leq 0$ into (\ref{gamma}) yields $\gamma_k \in [0,1]$.
\hfill \qed
\end{proof}

\begin{lemma}[\cite{yang24}]
Assume that $\gamma_k$ is selected as of (\ref{gamma}), and $\E^0\succ 0$. Then $\E^{k}\succ 0$ and $(\E^{k})^{-1}\succ 0$ for all $k >0$.
	\label{succE}
\end{lemma}
\begin{proof}
In view of the first inequality of (\ref{conditions}), it must have $\z^{k^{\T}} \s^k >0$. Since $(\E^0)^{-1} =\I \succ 0$, equation (\ref{E1}) implies that  $(\E^{k})^{-1}\succ 0$ for all $k >0$ because of the following reason: if $\v \neq 0$ and $\v^{\T} \s^k =0$, then  $\v^{\T}(\E^{k+1})^{-1}\v=\v^{\T}(\E^{k})^{-1}\v >0$; if $\v \neq 0$ and $\v^{\T} \s^k \neq 0$, since $\frac{\v^{\T}\s^{k} \s^{{k}^{\T}}\v}{\z^{{k}^{\T}} \s^{k}}>0$ due to Corollary \ref{zksk}, it follows $\v^{\T}(\E^{k+1})^{-1}\v>0$. Therefore, $(\E^{k+1})^{-1} \succ 0$, so is $\E^{k+1}\succ 0$.
\hfill \qed
\end{proof}

In view of Lemma~\ref{succE}, pre-multiplying $\nabla f(\x^{k})^{\T} \neq \0$ on both sides of (\ref{calMBFGS}) yields the following
\begin{lemma}
	The search direction $\d^k \neq \0$ of (\ref{calMBFGS}) is a descent direction for every $k>0$.
	\label{descent}
\end{lemma}

Besides the inequalities in (\ref{conditions}), the Wolfe conditions also play a critical role in establishing the global convergence of the arc-search BFGS algorithm. Therefore, we show that the arc-search BFGS method satisfies the Wolfe conditions in a natural way. In the following discussion, we use the order notaion
$o(\cdot)$ described in \cite[Page 591]{nocedal99}. Let $\{ \eta_k \}$ and $\{ \nu_k \}$ be two scalar nonnegative infinite sequences, 
We say $\eta_k=o(\nu_k)$ if the ratio $\{ \eta_k / \nu_k \}$ approaches zero, that is 
\begin{equation}
\lim_{k \rightarrow \infty} \frac{\eta_k}{\nu_k} =0.
\label{littleO}
\end{equation}
We refer $o$ as to little o. 

%

\begin{lemma}
Assume $-\infty < \nabla f(\x^k)^{\T} \d^k < 0$ for every iteration $k$. Then there is a small $\alpha_k>0$ in the arc-search algorithm satisfies the Wolfe condition defined in (\ref{WolfeC}) with $0 < \sigma_1 < \sigma_2<1$.
\label{arcWolfe}
\end{lemma}
\begin{proof}
By the definition of little o, for $\alpha$ small enough, there is a positive constant $0<p_2$ such that
\begin{equation}
p_2 \alpha \nabla f(\x^{k})^{\T} \d^k \leq o(\alpha) \leq -p_2 \alpha \nabla f(\x^{k})^{\T} \d^k.
	\label{Ltmp1}
\end{equation}
Since 
\begin{equation}
\frac{1}{2}\sin^2(\alpha) \le (1-\cos(\alpha)) \le \sin^2(\alpha),
\label{sincos}
\end{equation}
\begin{eqnarray} 
f(\x^{k+1}) & =&  f(\x^k+\sin(\alpha_k) \d^k -(1-\cos(\alpha_k)) \dot{\d}^k)   \nonumber  \\
& = & f(\x^k) + \nabla f(\x^k)^{\T} (\x^{k+1}-\x^k) + o(\| \x^{k+1}-\x^k \| )    \nonumber  \\
	& = & f(\x^{k}) + \nabla f(\x^k)^{\T} [ \sin(\alpha_k) \d^k -(1-\cos(\alpha_k)) \dot{\d}^k] + o(\alpha_k)  \nonumber  \\
	& = & f(\x^k) + \left[  \sin(\alpha_k)\nabla f(\x^k)^{\T} \d^k - (1-\cos(\alpha_k))  \nabla f(\x^k)^{\T}  \dot{\d}^k  \right] + o(\alpha_k).
\end{eqnarray}	
The magnitude of the last item in the square bracket is in the order of $o(\alpha_k)$. To simplify our notation, we will abuse the notation of $o(\alpha_k)$ and denote all higher-order terms of $\alpha$ as $o(\alpha_k)$. Then in view of (\ref{Ltmp1}), for small $\alpha_k$, $f(\x^{k+1}) $ can be written as
\begin{eqnarray}
f(\x^{k+1}) 	& = & f(\x^k) +\sin(\alpha_k)  \nabla f(\x^k)^{\T} \d^k + o(\alpha_k)  \nonumber  \\
	& \leq &  f(\x^k) +\alpha_k \nabla f(\x^k)^{\T} \d^k -p_2 \alpha_k \nabla f(\x^k)^{\T} \d^k	 \nonumber  \\
	& \leq &  f(\x^k) +(1-p_2)\alpha_k \nabla f(\x^k)^{\T} \d^k    \nonumber  \\
	& = & f(\x^k) + \sigma_1 \alpha_k \nabla f(\x^k)^{\T} \d^k. 
\end{eqnarray}

Similarly, for small $\alpha$, $\nabla f(\x^{k+1})$ can be written as
\begin{eqnarray} 
	\nabla f(\x^{k+1})  & =& \nabla f(\x^{k}) +\H ( \x^{k+1} - \x^k) + o(\| \x^{k+1}-\x^k \| )    \nonumber  \\
	& = &  \nabla f(\x^{k}) +\H  [ \sin(\alpha_k) \d^k -(1-\cos(\alpha_k)) \dot{\d}^k]  + o(\alpha_k)  \nonumber  \\
	& = & \nabla f(\x^{k})  + \sin(\alpha_k) \H   \d^k   + o(\alpha_k),
	\label{wolfe2}
\end{eqnarray}
Pre-multiplying $\d^{k^{\T}}$ to both sides of (\ref{wolfe2}) yields
\begin{eqnarray} 
	\d^{k^{\T}} \nabla f(\x^{k+1})
	= \d^{k^{\T}}\nabla f(\x^{k})   + \sin(\alpha_k) \left( \d^{k^{\T}} \H \d^k \right)  + o(\alpha_k).
	\label{wolfe3}
\end{eqnarray}
For $\alpha_k>0$ small enough, since $\d^{k^{\T}} \H \d^k > 0$ and $\lim_{\alpha_k \rightarrow \infty} \frac{\sin(\alpha_k)}{\alpha_k}=1$, there exists a $p_3>0$, such that 
\begin{equation}
\sin(\alpha_k)  \d^{k^{\T}} \H \d^k  +  o(\alpha_k)  > -\alpha_k p_3 \d^{k^{\T}}\nabla f(\x^{k})  ,
\end{equation}
therefore,
\begin{eqnarray} 
\d^{k^{\T}}  \nabla f(\x^{k+1})	\geq (1-\alpha_k p_3) \d^{k^{\T}}\nabla f(\x^{k}) = \sigma_2  \d^{k^{\T}} \nabla f(\x^{k}).
	\label{wolfe4}
\end{eqnarray}
Clearly, we can select $p_3$ small enough such that $\sigma_1 < \sigma_2<1$. 
\hfill \qed
\end{proof}


Now we are ready to discuss the global convergence result for the arc-search BFGS algorithm. Denote 
\begin{equation}\label{cosTheta}
  \cos(\theta_k)=-\frac{\nabla f(\x^{k})^{\T} \d^k}{\| \nabla f(\x^{k}) \| \| \d^k \|},
\end{equation} 
which measures the quality of the search direction.
Our global convergence proof is based on a modified (arc-search) Zoutendijk theorem \cite{zoutendijk70} which can be stated as follows.
\begin{theorem}
	Suppose that $f$ is bounded below in $\R^n$ and that $f$ is continuously twice differentiable in an open set ${\cal M}$ containing the level set ${\cal L}=\{ \x: f(\x) \le f(\x_0) \}$. Assume that the gradient is Lipschitz continuous on 	${\cal M}$, i.e., there exists a constant $L>0$ such that
	\begin{equation}
		\|\nabla f(\x)-\nabla f(\y) \| \le L \| \x-\y\|,
	\end{equation}
	for all $\x, \y \in {\cal M}$. Assume further that $\d^k$ is a 	descent direction and a small $\alpha_k$ ($\frac{ \| \d^k \|}{ \| \dot{\d}^k \|} \geq \alpha_k >0$) satisfies the Wolfe condition. Then
	\begin{equation}
		\sum_{k \ge 0} \cos^2(\theta_k) \| \nabla f(\x^{k}) \|^2 < \infty.
\label{zou}
	\end{equation}
	\label{Zoutendijk}
\end{theorem}
\begin{proof}
In view of (\ref{C2}), we have
\begin{equation}
(\nabla f(\x^{k+1}) -\nabla f(\x^{k}))^{\T} \d_k \geq (\sigma_2 -1) \nabla f(\x^{k})^{\T} \d_k.
\label{C2T}
\end{equation}
Using Lipschitz condition (\ref{function}), (\ref{updateX}), and (\ref{sincos}) yields
\begin{eqnarray}
(\nabla f(\x^{k+1}) -\nabla f(\x^{k}))^{\T}  \d^k  & \leq &  \| \nabla f(\x^{k+1}) -\nabla f(\x^{k}) \| \| \d^k \|  \nonumber \\
  & \leq &  L \| \x^{k+1} - \x^k \| \| \d^k \|   \nonumber \\
  & \leq &  L \| \sin(\alpha_k) \d^k -(1-\cos(\alpha_k)) \dot{\d}^k \|  \| \d^k \|  \nonumber \\
  & \leq &  L \left( \sin(\alpha_k)  \| \d^k\| + \sin^2(\alpha_k) \| \dot{\d}^k \| \right) \| \d^k\|  \nonumber \\
  & \leq &  \alpha_k L ( \| \d^k \|^2+  \alpha_k \| \dot{\d}^k \|  \| \d^k\| ) .
\label{arcL}
\end{eqnarray}
For $\alpha_k \leq \frac{ \| \d^k \|^2}{ \| \dot{\d}^k \|  \| \d^k\|}$, inequality (\ref{arcL}) can be written as (if $\| \dot{\d}^k \| = 0$, according to (\ref{arcL}), the following inequality still holds)
\begin{eqnarray}
(\nabla f(\x^{k+1}) -\nabla f(\x^{k}))^{\T}  \d^k   \leq  2\alpha_k   L  \| \d^k \|^2.
\label{arcL1}
\end{eqnarray}
In view of (\ref{C2T}), this yields
\begin{equation}
\alpha_k \geq \frac{(\nabla f(\x^{k+1}) -\nabla f(\x^{k}))^{\T}  \d^k}{2L  \| \d^k \|^2} \geq \frac{(\sigma_2 -1) \nabla f(\x^{k}))^{\T} \d^k}{2L  \| \d^k \|^2}.
\end{equation}
Substituting this inequality into (\ref{C1}) and using (\ref{cosTheta}) yield
\begin{eqnarray}
 f(\x^{k+1})  &  \leq   &  f (\x^k) + \sigma_1 \alpha_k \nabla f(\x^{k})^{\T}  \d^k \nonumber \\
   & \leq &  f (\x^k)  - \frac{\sigma_1(1-\sigma_2)}{2L} \frac{(\nabla f(\x^{k})^{\T}  \d^k)^2}{\| \d^k \|^2} \nonumber \\
   & = &  f (\x^k)  - C \cos^2(\theta_k) \| \nabla f(\x^{k}) \|^2   \nonumber \\
   & \leq &  f (\x^0)  - C \sum_{j=0}^k \cos^2(\theta_j) \| \nabla f(\x^{j}) \|^2
\label{importantIneq}
\end{eqnarray}
where $C=\frac{\sigma_1(1-\sigma_2)}{2L}$. Since $f$ is bounded below, inequality (\ref{importantIneq}) implies that (\ref{zou}) holds.
\hfill \qed
\end{proof}

\begin{remark}
The modified (arc-search) Zoutendijk theorem indicates that the convergence properties are directly related to the quality of the search direction $\cos(\theta_k)$. If $\d^k$ is a descent direction and $\cos(\theta_k) \ge \delta >0$, for all $k$, then the algorithm is globally convergent because $\lim_{k \rightarrow \infty} \nabla_{\x} f({\x^k})=\0$. This property will be used to prove the global and superlinear convergence of Algorithm~\ref{newAlgo}. 
\label{dfdxgx}
\end{remark}

The following theorem is the main result of \cite[Theorem 8.5]{nocedal99}.

\begin{theorem}\cite[Theorem 8.5]{nocedal99}
  Let $\s^k$ and $\z^k$ be generated by Algorithm \ref{newAlgo} which meet the inequalities of (\ref{conditions}), i.e.,
  \begin{equation}
  m \le \frac{\z^{k^{\T}} \s^k}{ \s^{k^{\T}} \s^k} \,\,\,\,\, \text{and} \,\,\,\,\,
\frac{ \z^{k^{\T}} \z^k}{ \z^{k^{\T}} \s^k} \le M
  \label{repeatM}
\end{equation}
Then, there exists a subsequence indices $\{ j_k \}$ such that
\begin{equation}
\cos(\theta_{j_k}) \ge \delta > 0.
\label{subseq}
\end{equation}
\label{nocedal}
\end{theorem}

\begin{remark}
In view of Lemma \ref{conditionM}, the conditions (\ref{repeatM}) of Theorem \ref{nocedal} are always satisfied if $\gamma_k$ is chosen according to~(\ref{gamma}).
\end{remark}

\begin{theorem}
Suppose that $f$ is bounded below in $\R^n$, $f$ is continuously twice differentiable in an open set ${\cal M}$ containing the 
level set ${\cal L}=\{ \x: f(\x) \le f(\x^0) \}$, and Assumptions 1-4 hold. In addition, if $\| \s^k \|=o(1/k^{1+\epsilon})$ for any $\epsilon >0$, then the sequence generated by Algorithm \ref{newAlgo} is globally convergent in the sense that $\lim \inf \|  \nabla f(\x^{k}) \| \rightarrow 0$. Moreover, the sequence generated by Algorithm \ref{newAlgo} converges to some local minimum point $\bar{\x}$ satisfying $\|  \nabla f(\bar{\x}) \|=0$ with superlinear rate.
\end{theorem}
\begin{proof}
Let $\s^k$ and $\z^k$ be generated by Algorithm \ref{newAlgo}. It is proved in Lemma~\ref{conditionM} that inequalities of (\ref{conditions}) hold for all $k \ge 0$. It follows from Theorem~\ref{nocedal} that there is a subsequence indices $\{ j_k \}$ such that (\ref{subseq}) holds.
Let $\d^k$ be the search direction generated by Algorithm \ref{newAlgo}. According to Lemma~\ref{descent}, $\d^k$ is a descent direction. In view of Lemma \ref{arcWolfe}, there is an $\alpha_k$ in Algorithm~\ref{newAlgo} that satisfies the Wolfe conditions. Applying these conditions to Theorem~\ref{Zoutendijk}, we have 
\begin{equation}
\sum_{k \in  \{ j_k \}} \cos^2(\theta_k) \| \nabla f(\x^{k})  \|^2 < \infty.
\label{liminf}
\end{equation}
Using the notation of $\lim \inf$ defined in \cite[Page 578]{nocedal99}, 
inequalities (\ref{subseq}) and (\ref{liminf}) imply
\begin{equation}
\lim \inf \| \nabla f(\x^{k})  \| \rightarrow 0.
\label{limit}
\end{equation}
This shows that the sequence generated by Algorithm \ref{newAlgo} is globally convergent.

Denote the $i$-th component of $\s^k$ by  $s_i^k$ and the $i$-th component of $\x^k$ by  $x_i^k$. Since $\| \s^k \|=o(1/k^{1+\epsilon})$, we have $| s_i^k |=o(1/k^{1+\epsilon})$. For the $i$-th component $x_i^{k+1}$ of $\x^{k+1}$, we have $x_i^{k+1}=x_i^0+\sum_{j=0}^{k} s_i^j$. To show that $\{ x_i^{k+1} \}$ is convergent, we just need to show $\lim_{k \rightarrow \infty} \sum_{j=0}^{k} s_i^j$ is convergent. The latter is guaranteed if 
\begin{equation}
\lim_{k \rightarrow \infty} \sum_{j=0}^{k} | s_i^j|
\label{absoluteC}
\end{equation}
is convergent. Since $| s_i^k |=o(1/k^{1+\epsilon})$ holds, it follows that every component $\{ x_i^{k} \}$ of $\{ \x^{k} \}$ is convergent. Therefore, this shows that $\x^{k} \rightarrow \bar{\x}$ and $\| \nabla f(\bar{\x}) \| = 0$. 
In view of the Assumption 2, the function is locally strongly convex in a neighborhood of 
$\bar{\x}$ which satisfies $\| \nabla f(\bar{\x}) \| = 0$, this means that $\bar{\x}$ is an isolated local minimizer in the neighborhood. This shows that the iterates will converge to a local minimizer. Since Algorithm \ref{newAlgo} is globally convergent and the convergent point is a local minimizer, for sufficiently large $k$ and for $\delta$ defined in Assumption 2, we have $f(\x^k) \le f(\bar{\x}) +\delta$. Therefore, for all $\v \in \R^n$ and $k \geq 0$,
\begin{equation}
m \|\v\|^2 \le \v^{\T} \H(\x^k) \v \le M \| \v\|^2.
\label{convexk}
\end{equation}
Using Corollary \ref{corollaryTaylor} 
\begin{equation}
\y^k=\nabla f(\x^{k+1}) -\nabla f(\x^{k})=\int_{0}^{1} \H(\x^k+t\p^k) \s_k\mathrm{d}t \equiv \bar{\H}_k\s_k,
\label{taylor1}
\end{equation}
then Lemma \ref{ykskp0} implies that $\bar{\H}_k$ is positive definite in a small neighborhood of $\bar{\x}$, i.e., for $k$ large enough, for all $\v \in \R^n$,
\begin{equation}
m \|\v\|^2 \le \v^{\T} \bar{\H}_k \v \le M \|\v\|^2.
\label{convexbar}
\end{equation}
This means for $k$ large enough, 
\begin{equation}
\frac{\y_k^{\T}\s_k}{\s_k^{\T}\s_k}=\frac{\s_k^{\T}\bar{\H}_k \s_k}{\s_k^{\T}\s_k}=\frac{\s_k^{\T}\bar{\H}_k \s_k}{\| \s_k \|^2}=\frac{\s_k^{\T}}{\|\s_k\|}\bar{\H}_k \frac{\s_k}{\|\s_k\|} \ge m,
\label{Lbound}
\end{equation}
and
\begin{equation}
\frac{\y_k^{\T}\y_k}{\y_k^{\T}\s_k}=\frac{\s_k^{\T}\bar{\H}_k^2\s_k}{\s_k^{\T}\bar{\H}_k \s_k}
=\frac{(\bar{\H}_k^{\frac{1}{2}}\s_k)^{\T}}{\|\bar{\H}_k ^{\frac{1}{2}}\s_k\|}\bar{\H}_k 
\frac{\bar{\H}_k^{\frac{1}{2}}\s_k}{\|\bar{\H}_k^{\frac{1}{2}}\s_k\|}
 \le M.
\label{Ubound}
\end{equation}
In view of Lemma \ref{conditionM}, (\ref{gamma}), and Algorithm~\ref{newAlgo}, inequalities (\ref{Lbound}) and (\ref{Ubound}) imply that $\gamma_k =0$. Hence, the arc-search BFGS algorithm reduces to the classical BFGS algorithm when $k$ is sufficiently large. Finally, if Assumptions~2 and 4 hold, the classical BFGS converges at a superlinear rate \cite[Theorem 8.6]{nocedal99}. Therefore, the arc-search BFGS also converges at the superlinear rate since it coincides with the classical BFGS method for all sufficiently large $k$. 
\hfill \qed
\end{proof}

\section{Implementation details}\label{implementationDetails}

Algorithm \ref{newAlgo} has been implemented in MATLAB through the function {\tt arcBFGS}. This section provides detailed implementation aspects of the algorithm that were not fully discussed in Section \ref{convergenceAnalysis}, or that differ slightly from the description in Algorithm \ref{newAlgo}. In addition, this section presents numerical results on the standard CUTEst benchmark set problems and compares the proposed algorithm with several established methods.

\vspace{0.15in}
\subsection{Estimation of $m$, $M$, and lower bound of $\gamma_k$}

First, determining $m$ and $M$, which involves the computation of smallest and largest eigenvalues of $\E^k$, is expensive and may vary significantly from one iterate $\x^k$ to another. Moreover, these values may fail to satisfy the condition (\ref{convex}) unless the iterate is sufficiently close to a local minimizer. Nevertheless, $m$ and $M$ play important roles in the computation of $\gamma_k$ through (\ref{gamma0}) and (\ref{gamma2}). The resulting value of $\gamma_k$ directly affects the computations of $\z_k$ via (\ref{zk}), $(\E^{k+1})^{-1}$ via (\ref{E1}) , and ultimately the next search direction $\d^{k+1}$ via (\ref{calMBFGS}). Therefore, obtaining suitable estimates of $m$ and $M$ is crucial for the overall performance of the algorithm. 

The appropriate values of $m$ and $M$ satisfying (\ref{conditions}) depend on the specific function being optimized. To ensure robustness, one may choose a very small $m$ and a very large $M$ so as to cover a broad range of problems. However, such conservative choices may lead to poor numerical performance because inaccurate $m$ and $M$ can produce an unsuitable $\gamma_k$, which in turn might yield a poorly conditioned $(\E^k)^{-1}$ (see (\ref{Ekp1})). Consequently, the resulting  search direction $\d^k$ (see (\ref{dirMBFGS})) may become unreliable. On the other hand, choosing $m$ too large, and/or choosing $M$ too small may violate condition (\ref{convex}) as the iterates approach to a minimizer. 

As a compromise, our default parameter choices are $m=0.00001$ and $M=100000$, which generally produce very satisfactory computational results. However, in some cases, we observed that these fixed choices may lead to slow convergence requiring many iterations. To alleviate this issue, we fix the ratio of $M/m=10^{10}$ while allowing the middle point of between $m$ and $M$ to vary adaptively. By shifting the interval determined by $m$ and $M$, we aim to enlarge the feasible interval of $\gamma_k$ given in (\ref{gamma}), thereby providing greater flexibility in the selection of $\gamma_{k}$. In particular, this strategy increases the likelihood that $\gamma_k=0$ can be selected when the iterates are close to a minimizer, which is important for achieving a local superlinear convergence rate.

In view of (\ref{gamma}), we denote
\begin{equation}
	 \gamma_l = \begin{cases}
	  \max \{ 0, \check{\gamma}_k, \underline{\gamma}_k \}  & 	\text{if }  \s_k^{\T}\s_k > \y_k^{\T}\s_k , \\
	 0  & \text{if }  \s_k^{\T}\s_k = \y_k^{\T}\s_k , \\
	 \max \{ 0, \underline{\gamma}_k \} & \text{if }  \s_k^{\T}\s_k < \y_k^{\T}\s_k, 
	 \end{cases}
\label{lowerB}
\end{equation}
and $\gamma_u = 1$ the lower and upper bounds of $\gamma_k \in [\gamma_l, \gamma_u]$ respectively. Since the upper bound is fixed with $\gamma_u = 1$, our goal is to adjust the center between $m$ and $M$ to reduce the lower bound $\gamma_l$.

It follows from (\ref{gamma0}) and (\ref{gamma2}) that $m$ only affects the value of $\check{\gamma}_k$, and $M$ only affects the value of $\underline{\gamma}_k$. Carefully examining the proof of Lemma~\ref{conditionM} indicates (a) if $\s^{k^{\T}}\s^k > \y^{k^{\T}}\s^k$, then $\check{\gamma}_k < 1$, and $\check{\gamma}_k$ decreases as $m \in (0,1)$ decrease; (b) if $\s^{k^{\T}}\s^k < \y^{k^{\T}}\s^k$, in view of (\ref{lowerB}), the lower bound of $\gamma_l$ depends only on $\underline{\gamma}_k$; and (c) $\underline{\gamma}_k <1$, and $\underline{\gamma}_k$ decreases as $M \in (1,\infty)$ increases. We would like to move the center between $m$ and $M$ to decrease $\gamma_{l}$ toward zero to maximize the range of $\gamma_k$. Therefore, in view of (\ref{lowerB}), we consider three cases: (i) if $\s^{k^{\T}}\s^k < \y^{k^{\T}}\s^k$ and $\underline{\gamma}_k >0$, according to (\ref{gamma0}), we need to increase $M$ (and $m$ to keep the ration of $M$ to $m$) to decrease $\underline{\gamma}_k$ (in this way, when the iterate is close to a minimizer, a conventional BFGS may be selected); (ii) if $\s^{k^{\T}}\s^k > \y^{k^{\T}}\s^k$ and $\check{\gamma}_k > \underline{\gamma}_k$, we need to decrease $m \in [0,1)$ (and $M$ to keep the ratio) to decrease $\check{\gamma}_k$ (and increase $\underline{\gamma}_k$) so that $\max \{ 0,\check{\gamma}_k,\underline{\gamma}_k \}$ is reduced; and (iii) if $\s^{k^{\T}}\s^k > \y^{k^{\T}}\s^k$ and $\check{\gamma}_k < \underline{\gamma}_k$, we need to increase $M$ (and $m$ to keep the ration of $M$ to $m$) to decrease $\underline{\gamma}_k$ (and increase $\check{\gamma}_k$) so that $\max \{ 0,\check{\gamma}_k,\underline{\gamma}_k \}$ is reduced.

The above simple heuristics to dynamically adjust $m$ and $M$ are summarized as the following algorithm.

\begin{algorithm} {Selection of $m$ and $M$ } \\
\begin{algorithmic}[1] 
\STATE Set $m=\bar{m}=10^{-5}$, $M=\bar{M}=10^5$, and calculate $\check{\gamma}_k$ and $\underline{\gamma}_k$.
\IF{$\s^{k^{\T}}\s^k < \y^{k^{\T}}\s^k$ and $\underline{\gamma}_k >0$}
	\FOR{$i=1:4$}
		\STATE   $M=10^i \bar{M}$ and $m=10^i \bar{m}$
		\STATE   Calculate $\underline{\gamma}_k$ and set $\gamma_l = \max \{ 0,  \underline{\gamma}_k$ \}
		\IF{$\underline{\gamma}_k < 0$}
			\STATE Break `for loop'
		\ENDIF
	\ENDFOR
\ELSIF{$\s^{k^{\T}}\s^k > \y^{k^{\T}}\s^k$ and $\check{\gamma}_k>\underline{\gamma}_k$}
	\FOR{$i=1:4$}
		\STATE   $M=10^{-i} \bar{M}$ and $m=10^{-i} \bar{m}$
		\STATE   calculate $\underline{\gamma}_k$ and $\check{\gamma}_k$
		\IF{$\check{\gamma}_k>\underline{\gamma}_k$}
			\STATE $\gamma_l =  \max \{ 0, \check{\gamma}_k$ \}
		\ELSE
			\STATE $\gamma_l =   \max \{ 0,  \underline{\gamma}_k \}$ and break `for loop'
		\ENDIF
	\ENDFOR
\ELSIF{$\s^{k^{\T}}\s^k > \y^{k^{\T}}\s^k$ and $\check{\gamma}_k<\underline{\gamma}_k$}
		\FOR{$i=1:4$}
		\STATE   $M=10^i \bar{M}$ and $m=10^i \bar{m}$
		\STATE   calculate $\underline{\gamma}_k$  and $\check{\gamma}_k$
		\IF{$\check{\gamma}_k<\underline{\gamma}_k$}
			\STATE $\gamma_l =  \max \{ 0, \underline{\gamma}_k \}$
		\ELSE
			\STATE $\gamma_l =   \max \{ 0, \check{\gamma}_k \}$ and break `for loop'
		\ENDIF
	\ENDFOR
\ELSIF{$\s^{k^{\T}}\s^k > \y^{k^{\T}}\s^k$ and $\check{\gamma}_k=\underline{\gamma}_k$}
	\STATE $\gamma_l =  \max \{ 0,  \check{\gamma}_k \}$
\ELSIF{$\s^{k^{\T}}\s^k = \y^{k^{\T}}\s^k$}
	\STATE $\gamma_l = 0$
\ENDIF
\end{algorithmic}
\label{mMAlg}
\end{algorithm}

This algorithm is applied between Lines~\ref{step8} and \ref{step9} of Algorithm \ref{newAlgo}, and is executed before the calculation of (\ref{gamma}). This strategy significantly reduces the number of iterations for the problems which had slow convergence when fixed $m$ and $M$ were used. Moreover, it has little impact on the remaining problems.

\subsection{Select of $\gamma_k$}

For some test problems, selecting the smallest admissible value $\gamma_k=\gamma_l$ during the early iterations may lead to a poorly conditioned estimate of $\E^k$ (see (\ref{Ekp1})). As a consequence, the resulting search direction $\d^k$ can become excessively large, while the step size $\alpha_k$ becomes very small. This situation has two undesirable effects: (a) the search direction $\d^k$ may not be accurate, and (b) the estimate of $\dot{\d}^k$ may become unreliable due to a large numerical error. In such cases, we prefer to maintain a well-conditioned matrix $\E^k$ so that a reliable search direction $\d^k$ can be obtained.  

Intuitively, a large value of $\| \nabla_{\x} f(\x^k) \|$ is an indicator that the current iterate is far way from a minimizer. Let $g_k=\| \nabla_{\x} f(\x^k) \|$ denote the norm of the gradient at current iterate, and $g_{max}=\max_{k} \| \nabla_{\x} f(\x^k) \|$ denote the maximum of $ \| \nabla_{\x} f(\x^k) \|$ up to current iterate. The following formula is used to select $\gamma_k \in [\gamma_l,\gamma_u]$.
\begin{equation}
	\gamma_k = \begin{cases}
		\gamma_{l} & \text{ if } \frac{g_k}{g_{max}} \leq 10^{-2} \text{ or } g_k \leq 100 \\
		\gamma_{l} + \frac{g_k}{g_{max}}(\gamma_u - \gamma_{l}) & \text{ if } 10^{-2} < \frac{g_k}{g_{max}} < 1 \\
		1  & \text{ if } 1 \leq \frac{g_k}{g_{max}} .
		\end{cases}
\label{gamma_adj}
\end{equation}

\subsection{Handeling incorrect $\dot{\d}^k$}

As we mentioned in the previous section, the estimation of $\dot{\d}^k$ using~(\ref{ddd}) or~({\ref{ddd1}) can be inaccurate if $\d^k$ is excessively large due to numerical errors while $\alpha_k$ is very small. In addition, if $\gamma_k=0$, according to Algorithm~\ref{newAlgo} step~\ref{step12}, we must set $\dot{\d}^k=\0$ to guarantee superlinear convergence rate. Therefore, $\dot{\d}$ is adjusted from the value computed by (\ref{ddd}) or (\ref{ddd1}) according to the following rule.
\begin{equation}
	\dot{\d}^k = \begin{cases}
		\0 & \text{ if } \gamma_k =0, \\
		0.2 \frac{ \| \d^k \| }{ \| \dot{\d}^k \| } \dot{\d}^k & \text{ if } \| \dot{\d}^k \| >10 \| \d^k \|,  \\
		\dot{\d}^k &   \text{ else.} 
		\end{cases}
\label{dotdk_adj}
\end{equation}
Although combining (\ref{ddd}) with (\ref{dotdk_adj}) (or combining (\ref{ddd1}) with (\ref{dotdk_adj})) significantly improves the performance of the algorithm, we believe that developing a more accurate approach for estimating $\dot{\d}^k$ remains an important topic for future research.

\section{Numerical test}\label{numericalTest}

The arc-search BFGS implementation {\tt arcBFGS} and
the BFGS algorithm implemented in Matlab Optimization 
Toolbox function {\tt fminunc} are tested using the 
{CUTEst} test problem set. {\tt fminunc} options are set as
\footnotesize
\[
\mbox{options = optimset('LargeScale','off','MaxFunEvals',1e+20,'MaxIter',
5e+5,'TolFun',1e-20, 'TolX',1e-10).}
\]
\normalsize
This setting is selected to ensure that the BFGS implementation in Matlab function {\tt fminunc} has enough iterations either to converge or to fail.  

We conducted numerical experiments for both {\tt arcBFGS} and {\tt fminunc} using the {\tt CUTEst} test problem set, downloaded from the Princeton test problem collection \cite{web}. Since the {\tt CUTEst} problems are provided in AMPL files, we first converted the AMPL files into nl-files so that the Matlab functions could read the {\tt CUTEst} models. The Matlab functions {\tt arcBFGS} and {\tt fminunc} were then used to read the nl-files and solve the corresponding test problems.

All function values and gradient are calculated from AMPL command $\mbox{[fxk1,gk1] = matampl(xk1)}$. Both {\tt arcBFGS} and {\tt fminunc} use these values within their optimization procedures. The tests employed the initial points provided by the {\tt CUTEst} test set. For each problem, we recorded the objective function values, the norms of the gradients at the final iterates, and the corresponding iteration counts. The numerical results are reported in Table~1. In the table, {\tt iter} denotes the total number of iterations performed by the algorithm, {\tt obj} denotes the objective function value at termination, and {\tt gradient} denotes the norm of the gradient at the final iterate. 

Because different algorithms are implemented in different programming languages, it is generally difficult to make completely fair comparisons based on computational time. In particular, Matlab is an interpreted language and is typically slower than compiled languages such as C/C++ or FORTRAN. Fortunately, the analysis in Remark~\ref{efficiency} shows that the computational cost per iteration of the arc-search BFGS method is nearly the same as that of the conventional BFGS method. Therefore, comparing iteration counts between the algorithms provides a meaningful measure of performance.

Table~1 includes the test results obtained by the classical BFGS algorithm implemention {\tt fminunc} in the Matlab optimization toolbox and the arc-search BFGS implementation {\tt arcBFGS}. Note that {\tt fminunc} is implemented as compiled code, whereas {\tt arcBFGS} is implemented as interpreted Matlab code. The stopping criterion for {\tt arcBFGS} was set to $\epsilon = 10^{-5}$.

\footnotesize
\begin{center}
\begin{longtable}{|c|r|r|r|r|r|r|r|}
\caption{Test result for problems in CUTEst \cite{web}, initial points are given in CUTEst} \\
\hline    
Problem & size & iter   &  obj    & gradient  & iter    & obj     & gradient  \\
        & n    & arcBFGS  &  arcBFGS  & arcBFGS     & fminunc & fminunc & fminunc    \\
\hline
arglina    & 100  &  1    & 100.0       & 1.316e-13  & 4  & 100.0      &  5.6241e-05       \\ \hline
bard       & 3  &  16   & 0.8214e-02       & 9.2845e-06  & 20 & 0.8214e-02      &  0.1158e-05     \\ \hline
beale      & 2  &  11   & 4.2555e-13    & 2.7682e-06 & 15 & 0.2400e-09   &  0.1392e-06     \\ \hline
biggs6     & 6  &  71   & 4.8821e-10    & 5.8848e-06  & 68 & 0.43162e-03      &  {\bf 0.28136e-01} \\ \hline
box3       & 3  &  16   & 1.8263e-11    & 1.5785e-06   & 24 & 0.3880e-10   &  0.2364e-05     \\ \hline
brkmcc     & 2  &  5    & 0.1690e-00    & 2.6556e-08  & 5  & 0.1690e-00      &  0.4542e-06     \\ \hline
brownal    & 10  &  6   & 9.4808e-14   & 6.5791e-07 & 16 & 0.3050e-08   &  {\bf 0.1043e-03}      \\ \hline
brownbs    & 2  &  38  & 1.2622e-29    & 7.1054e-09  & 11 & 0.9308e-04      &  {\bf 15798.5950e-00} \\ \hline
brownden   & 4  &  22  & 85822.2016e-00    &1.4567e-05  & 32 & 85822.2017e-00 &  {\bf 0.4646e-00} \\ \hline
chnrosnb   & 50  &  160  & 4.833e-14    & 5.5823e-06  & 96 & 30.2671e-00      &  {\bf 13.4381e-00} \\ \hline
cliff      & 2  &  5   & 0.19979e-00   & 5.2945e-06  & 1  & 1.0015e-00      &  {\bf 1.4147e-00} \\ \hline
cube       & 2  &  26   & 2.7048e-14    & 6.9257e-06 & 34 & 0.7987e-09   &  {\bf 0.1340e-03}       \\ \hline
deconvu    & 51  &  56   & 3.0599e-09    & 3.2062e-06   & 80 & 0.3158e-06   & {\bf  0.1750e-03}     \\ \hline
denschna   & 2  & 6    & 1.4673e-14    & 3.3458e-07  & 10 & 0.5000e-12   &  0.1581e-05     \\ \hline
denschnb   & 2  &  6    & 1.6576e-14   & 2.7736e-07  & 7  & 0.1000e-11   &  0.2200e-05     \\ \hline
denschnc   & 2  & 14    & 4.833e-12 & 9.6253e-06  & 21 & 0.1608e-08   &  {\bf 0.3262e-03}     \\ \hline
denschnd   & 3  &  76   & 3.4594e-09  & 3.2575e-06  & 23 & 45.2971e-00      &  {\bf 84.5851e-00} \\ \hline
denschnf   & 2  &  8   & 3.4868e-14   & 3.4354e-06  & 10 & 0.2000e-10   &  {\bf 0.1005028e-03}     \\ \hline
dixon3dq   & 10  &  14  & 5.0073e-14    & 5.2572e-07  & 20 & 0.1400e-11   &  0.3661e-05     \\ \hline
djtl       & 2  &  86  & -8951.5447e-00       & 0.0060653  & 3  & -8033.8869e-00      &  {\bf 1273.3319e-00} \\ \hline
eigenals   & 110  &  93   & 1.2617e-14    & 2.5416e-06  & 81 & 0.11393e-02   &  {\bf 0.092485e-00} \\ \hline
eigenbls   & 110  &  488  & 1.7596e-12    & 7.4625e-06  & 89 & 0.3676e-00      &  {\bf 0.75251e-00} \\ \hline
engval2    & 3  &  31   & 2.6758e-17    & 2.6573e-07  & 29 & 0.3953e-09   &    {\bf 0.2799e-03}     \\ \hline
errinros   & 81  &  46   & 2.4798e-07  & 6.6802e-06  & 92 & 0.4577e-03      &  {\bf 0.2553e-00} \\ \hline
expfit     & 2  &  11   & 0.2405e-00       & 2.402e-07 & 12 & 0.2405e-00      &  0.2263e-05     \\ \hline
extrosnb   & 10  &  1    & 0.0       & 0.0     & 1  & 0.0      &  0.0       \\ \hline
fletcbv2   & 100  &  96   & -0.5140e-00       & 8.0932e-06  & 98 & -0.5140e-00      &  {\bf 0.1087e-04}     \\ \hline
fletchcr   & 100  &  206  & 7.0761e-14  & 8.1494e-06  & 76 & 0.2699e-00       &  {\bf 3.2669e-00} \\ \hline
genhumps   & 5  &  65   & 1.9195e-10    &6.6494e-06  & 61 & 0.2239e-07   & {\bf 0.8180e-04}     \\ \hline
growthls   & 3  &  1    & {\bf 3542.1490e-00} & 0      & 12 & 12.4523e-00      &  {\bf 0.5809e-01} \\ \hline
hairy      & 2  &  52   & 20.0       & 4.0245e-06 & 22 & 20.0    &  {\bf  0.8435e-04}     \\ \hline
hatfldd    & 3  &  25   & 6.6151e-08   & 1.3139e-07  & 19 & 0.6615e-07   &  0.2355e-05     \\ \hline
hatflde    & 3  &  32   & 4.4345e-07    & 1.4845e-06  & 9  & 0.6210e-06   &  0.7970e-05     \\ \hline
heart6ls   & 6  & 848  & 3.537e-20  & 1.863e-06 & 53 & 0.6370e-00      &  {\bf 74.8787e-00} \\ \hline
helix      & 3  &  28 & 1.0859e-13 & 6.339e-06  & 29 & 0.2260e-10   &  {\bf 0.4196e-04}     \\ \hline
hilberta   & 10  &  28   & 2.3253e-07    & 5.7438e-06  & 35 & 0.2289e-06   &  0.3263e-05     \\ \hline
hilbertb   & 50  &  8   & 1.2136e-12    & 4.9428e-06  & 6  & 0.21e-11   &  0.6542e-5     \\ \hline
himmelbb   & 2  &7    & 3.4657e-09    & 5.6347e-06  & 6  & 0.1462e-04      &  {\bf 0.1251e-02}       \\ \hline
himmelbf   & 2  &  6    & 6.8141e-12   & 8.4837e-06 & 8  & 0.1000e-12   &  0.1448e-05     \\ \hline
himmelbg   & 2  &  6    & 6.8141e-12    & 8.4837e-06  & 8  & 0.1000e-12   &  0.1448e-05     \\ \hline
himmelbh   & 2  &  5   & -1.000e-00   & 5.8436e-07 & 7  & -0.9999e-00  &  0.2607e-06     \\ \hline
humps      & 2  &  67   & 2.063e-10   & 6.4277e-06  & 25 & 5.4248e-00      &  {\bf 2.3625e-00} \\ \hline
jensmp     & 2  &  1    & {\bf 2020}       & 0             & 16 & 124.3621e-00      & {\bf 0.2897e-03}    \\ \hline
kowosb     & 4  &  25   & 0.3075e-03    & 6.5807e-07  & 33 & 0.3075e-03   &  0.1253e-06     \\ \hline
loghairy   & 2  &  78   & 0.18232e-00       & 3.7674e-06e-06  & 11 & 2.5199e-00 &  {\bf 0.8033e-00}        \\ \hline
mancino    & 100  &  122   & 5.3267e-20    & 6.5609e-07  & 9  & 0.2204e-02      &  {\bf 1.2243e-00} \\ \hline
maratosb   & 2  &  2    & -1.0       &  5.2357e-09  & 2  & -0.9997e-00      &  {\bf 0.4976e-01} \\ \hline
mexhat     & 2  &  1    & -0.401e-01       & 6.7066e-06  & 4  & -0.4009e-01      &  {\bf 0.13703e-04}    \\ \hline
osborneb   & 11  &  49   & 0.040138       & 4.8859e-06  & 76 & 0.4013e-01      &  0.7884e-05     \\ \hline
palmer1c   & 8  &  33  &0.097598e-00     & 1.2491e-05  & 38 & 16139.4418e-00      &  {\bf 655.0159e-00} \\ \hline
palmer2c   & 8  &  38   & 0.14421e-01       & 5.4894e-06  & 60 & 98.0867e-00      &  {\bf 33.4524e-00} \\ \hline
palmer3c   & 8  &  38   & 0.19538e-01       & 4.2066e-08  & 56 & 54.3139e-00      &  {\bf 7.8518e-00} \\ \hline
palmer4c   & 8  &  36   & 0.50311e-01       &2.8175e-06  & 56 & 62.2623e-00      &  {\bf 6.6799e-00} \\ \hline
palmer5c   & 6  &  11   & 2.1281e-00    & 9.8678e-07  & 14 & 2.1280e-00      &  {\bf 0.7484e-03}       \\ \hline
palmer6c   & 8  &  60   & 0.16387e-01    & 8.0068e-06  & 43 & 18.0992e-00      &  {\bf 0.7851e-00} \\ \hline
palmer7c   & 8  &  44   & 0.60199e-00   & 1.3956e-06  & 28 & 56.9098e-00      &  {\bf 4.0268e-00} \\ \hline
palmer8c   & 8  &  50   & 0.15977e-00   & 2.2619e-06 & 49 & 22.4366e-00      &  {\bf 1.2983e-00} \\ \hline
rosenbr    & 2  &  23   & 1.3706e-14    & 3.6743e-06  & 36 & 2.8336e-11   &  2.6095e-5     \\ \hline
sineval    & 2  &  65   & 2.7999e-17    & 4.5792e-07 & 47 & 0.2212e-00      &  {\bf 1.2315e-00} \\ \hline
sisser     & 2  &  16  & 1.1792e-08   & 5.0998e-06  & 11 & 1.5409e-08   &  7.2827e-06     \\ \hline
tointqor   & 50  &  35   & 1175.4722e-00   & 9.067e-06  & 40 & 1175.4722e-00 &  8.6856e-05     \\ \hline
vardim     & 100  &  21   & 3.0876e-18    & 1.1093e-06  & 1  & 9.2547e-10  &  {\bf 0.03539e-00} \\ \hline
watson     & 31  &  46   & 2.4798e-07    & 6.6802e-06  & 90 &  0.1050e-02      &  {\bf 0.4875e-00} \\ \hline
yfitu      & 3  &  63   & 6.6697e-13    & 6.0455e-07  & 56 & 0.4371e-02      &  {\bf 11.7373e-00} \\ \hline
\end{longtable}
\label{BFGSvsmBFGS}
\end{center}
\normalsize

We summarize the comparison of the test result as follows:
\begin{itemize}
\item[1.] The arc-search BFGS function {\tt arcBFGS} converges, satisfying the termination condition $|g(x_k)| < 10^{-5}$, for all $64$ test problems except {\tt brownden} and {\tt djtl}. In contrast, the Matlab Optimization Toolbox BFGS function {\tt fminunc} converges for only $22$ out of the $64$ problems. For the problems {\tt brownden} and {\tt djtl}, the solutions obtained by {\tt arcBFGS} are still very close to the optimal solutions, whereas the solutions produced by {\tt fminunc} remain far from optimality. In fact, for $28$ out of the $64$ problems, the solutions obtained by {\tt fminunc} are far from the optimal solutions.
\item[2.] For the problems on which both {\tt arcBFGS} and {\tt fminunc} converge successfully, {\tt arcBFGS} generally requires fewer iterations than {\tt fminunc}. The only exceptions observed in our experiments are the problems {\tt hatfldd}, {\tt hatflde}, {\tt hilbertb}, and {\tt sisser}.
\item[3.] For two problems, {\tt growthls} and {\tt jensmp}, {\tt arcBFGS} converges to a local minimum, whereas {\tt fminunc} finds better local solutions.
\end{itemize}

An alternative way to demonstrate the superiority of the proposed algorithm is through the use of {\it performance profiles}, which were first introduced in \cite{ty96} and further analyzed in \cite{dm02}. Let ${\cal S}$ denote the set of solvers and ${\cal P}$ the set of test problems. Let $n_s$ and $n_p$ represent the number of solvers and test problems, respectively, and let $m_{p,s}$ denote the merit value (for example, the number of iterations) obtained by solver $s$ on problem $p$. The performance ratio is defined as \cite{ty96}
\begin{equation}
r_{p,s}= \frac{m_{p,s}}{\min \{ m_{p,s}: s \in {\cal S}  \}}.
\label{performRatio}
\end{equation}
The performance profiles for {\tt arcBFGS} and Matlab's {\tt fminunc} are presented in Figure~\ref{fig:profile1}. The results clearly demonstrate that the arc-search BFGS implementation {\tt arcBFGS} computes optimal solutions more efficiently than the classical BFGS implementation {\tt fminunc} for the tested problems.

\begin{figure}[ht]
\centerline{\epsfig{file=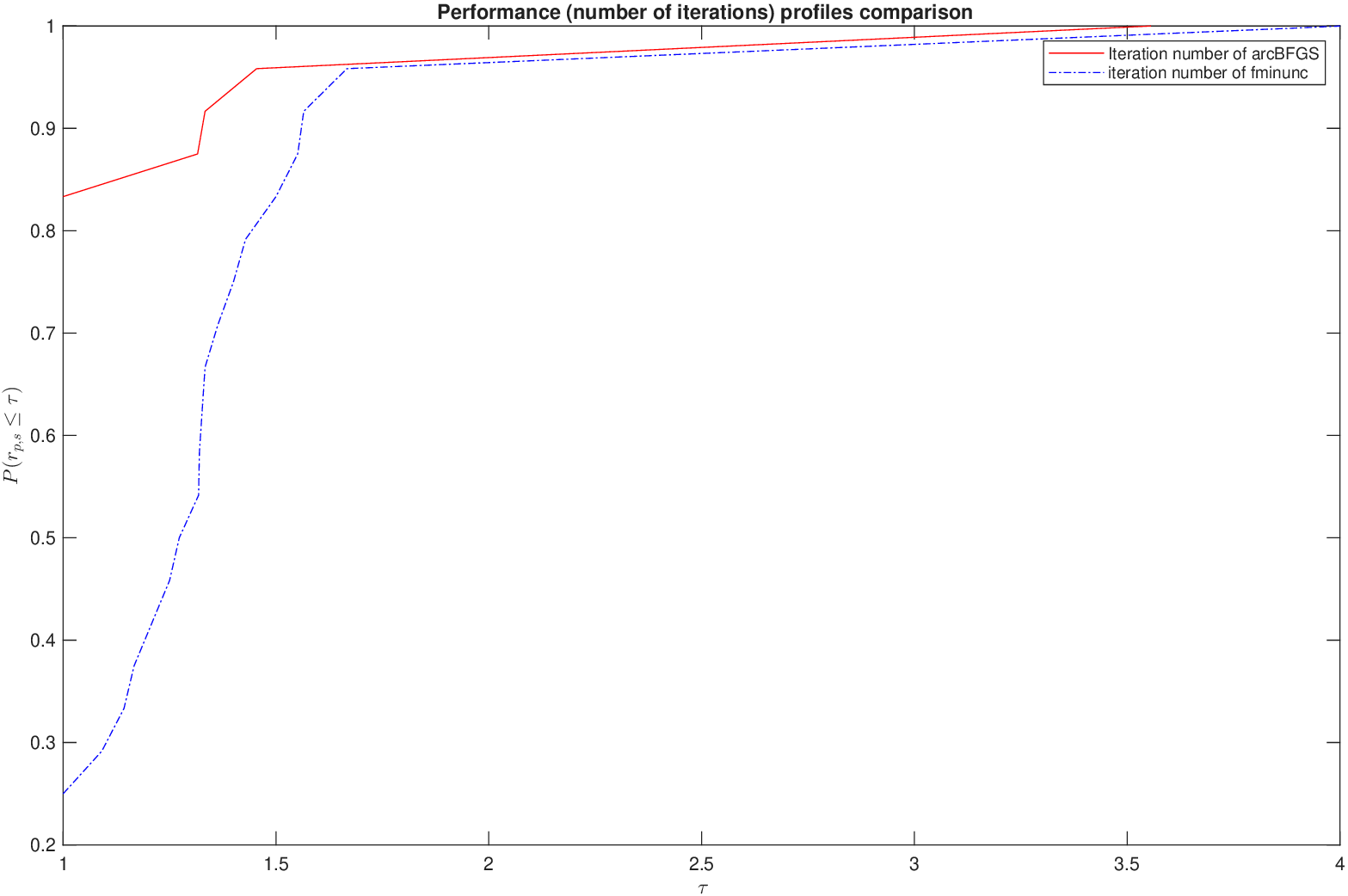,height=6cm,width=10cm}}
\caption{Performance profiles comparison for {\tt arcBFGS} and {\tt fminunc}.}
\label{fig:profile1}
\end{figure} 

Most of the above problems are also used by other 
researchers \cite{hager} to test some established and 
state-of-the-art algorithms. In \cite{hager}, 
CUTEst unconstrained problems are tested against 
limited memory BFGS algorithm \cite{nocedal80} 
(implemented as {\tt L-BFGS}), a descent and conjugate 
gradient algorithm \cite{hz05} (implemented as 
{\tt CG-Descent 5.3}), and a limited memory descent 
and conjugate gradient algorithm \cite{hz12} 
(implemented as {\tt L-CG-Descent}). 
We compare the test results obtained by our implementation 
of Algorithm \ref{newAlgo}, the robust BFGS \cite{yang24}, and the results obtained by 
algorithms \cite{hz12,hz05,nocedal80} (reported in 
\cite{hager}). For this test, we changed the stopping 
criterion to $\| g(x) \|_{\infty} \le 10^{-6}$
which is used in \cite{hager} for comparison between the 
results obtained in \cite{hz12,hz05,nocedal80}. The test 
results are listed in Table 2.

\footnotesize
\begin{center}
\begin{longtable}{|r|r|r|r|r|r|}
\caption{Comparison of arcBFGS, L-CG-Descent, L-BFGS, and CG-Descent 5.3 for 
problems in CUTEst \cite{web}, initial points are given in CUTEst}\\
\hline    
Problem      & size   & methods          & iter    &  obj    & gradient      \\
\hline
arglina      &  200   & mBFGS            &   1    &  {\bf 1.000e+002} &   1.894e-015  \\ 
             &        & arcBFGS          &   1    &  {\bf 1.000e+002} &   5.071e-014  \\
             &        & L-CG-Descent     &   1    &  2.000e+002       &   3.384e-008  \\
             &        & L-BFGS           &   1    &  2.000e+002       &   3.384e-008  \\
             &        & CG-Descent 5.3   &   1    &  2.000e+002       &   2.390e-007  \\
\hline
bard         &  3     & mBFGS            &   18   &   1.157e-001      &    9.765e-007           \\ 
             &        & arcBFGS          &   16   &   {\bf 8.215e-003}      &    7.067e-006           \\
             &        & L-CG-Descent     &   16   &   {\bf 8.215e-003}      &    3.673e-009           \\
             &        & L-BFGS           &   16   &   {\bf 8.215e-003}      &    3.673e-009           \\
             &        & CG-Descent 5.3   &   21   &   {\bf 8.215e-003}      &    1.912e-007           \\
\hline
beale        &  2     & mBFGS            &   13   &   4.957e-020      &    2.979e-010           \\ 
             &        & arcBFGS          &   11   &   2.665e-015      &    2.215e-007           \\
             &        & L-CG-Descent     &   15   &   2.727e-015      &    4.499e-008           \\
             &        & L-BFGS           &   15   &   2.727e-015      &    4.499e-008           \\
             &        & CG-Descent 5.3   &   18   &   1.497e-007      &    4.297e-007           \\
\hline
biggs6       &  6     & mBFGS            &   73   &   {\bf 7.777e-013}  &    4.920e-007         \\ 
             &        & arcBFGS          &   71   &   {\bf 3.407e-009}  &    9.581e-006         \\
             &        & L-CG-Descent     &   27   &   5.656e-003        &    2.514e-008         \\
             &        & L-BFGS           &   27   &   5.656e-003        &    2.514e-008         \\
             &        & CG-Descent 5.3   &   85   &   5.656e-003        &    9.195e-007         \\
\hline
box3         &  3     & mBFGS            &   21   &   1.692e-016        &    4.450e-008         \\ 
             &        & arcBFGS          &   16   &   2.765e-013        &    7.666e-007         \\
	         &        & L-CG-Descent     &   11   &   3.819e-013        &    7.584e-007         \\
	         &        & L-BFGS           &   11   &   3.819e-013        &    7.584e-007         \\
	         &        & CG-Descent 5.3   &   13   &   1.707e-010        &    6.003e-007         \\
\hline
brkmcc       &  2     & mBFGS            &   4    &    1.690e-001     &    8.034e-008           \\ 
             &        & arcBFGS          &   5    &    1.690e-001     &    5.049e-006           \\
             &        & L-CG-Descent     &   5    &    1.690e-001     &    6.220e-008           \\
             &        & L-BFGS           &   5    &    1.690e-001     &    6.220e-008           \\
             &        & CG-Descent 5.3   &   4    &    1.690e-001     &    5.272e-008           \\
\hline
brownal      &  200   & mBFGS            &   7     &    1.496e-016     &    8.518e-009           \\
             &        & arcBFGS          &   6     &    1.909e-012     &    3,648e-009           \\
             &        & L-CG-Descent     &   9     &    2.704e-018     &    4.540e-008           \\
             &        & L-BFGS           &   4     &    1.473e-009     &    6.663e-007           \\
             &        & CG-Descent 5.3   &   12    &    6.562e-011     &    1.392e-007           \\
\hline
brownbs      &  2     & mBFGS            &   632   &   1.952e-018      &   5.369e-007     \\ 
             &        & arcBFGS          &  38  &   3.008e-016      &   8.382e-006     \\
             &        & L-CG-Descent     &   13    &   0.000e+000      &   0.000e+000     \\
             &        & L-BFGS           &   13    &   0.000e+000      &   0.000e+000     \\
             &        & CG-Descent 5.3   &   16    &   1.972e-031      &   8.882e-010     \\
\hline
brownden     &  4     & mBFGS            &   21    &   8.582e+004      &   3.092e-010            \\ 
             &        & arcBFGS          &   22    &   8.582e+004      &   4.173e-006            \\
             &        & L-CG-Descent     &   16    &   8.582e+004      &   1.282e-007            \\
             &        & L-BFGS           &   16    &   8.582e+004      &   1.282e-007            \\
             &        & CG-Descent 5.3   &   38    &   8.582e+004      &   9.083e-007            \\
\hline
chnrosnb     &  50    & mBFGS            &   160   &  1.263e-015       &  3.525e-007             \\ 
             &        & arcBFGS          &   160   &  4.418e-014       &  7.632e-006            \\
             &        & L-CG-Descent     &   287   &  6.818e-014       &  5.414e-007             \\
             &        & L-BFGS           &   216   &  1.582e-013       &  5.565e-007             \\
             &        & CG-Descent 5.3   &   287   &  6.818e-014       &  5.414e-007             \\
\hline
cliff        &  2     & mBFGS            &   15    &   1.998e-001      &   7.602e-008            \\ 
             &        & arcBFGS          &   5    &   1.998e-001      &   2.518e-006            \\
             &        & L-CG-Descent     &   18    &   1.998e-001      &   2.316e-009            \\
             &        & L-BFGS           &   18    &   1.998e-001      &   2.316e-009            \\
             &        & CG-Descent 5.3   &   19    &   1.998e-001      &   6.352e-008            \\
\hline
cube         &  2     & mBFGS            &   21    &   4.231e-020      &   6.845e-010            \\ 
             &        & arcBFGS          &   26    &   1.721e-015      &   2.344e-006            \\
             &        & L-CG-Descent     &   32    &   1.269e-017      &   1.225e-009            \\
             &        & L-BFGS           &   32    &   1.269e-017      &   1.225e-009            \\
             &        & CG-Descent 5.3   &   33    &   6.059e-015      &   4.697e-008            \\
\hline
deconvu      &  61    & mBFGS            &   67    &   1.567e-009      &   9.999e-007            \\ 
             &        & arcBFGS          &   56   &   1.923e-007      &   8.780e-006            \\
             &        & L-CG-Descent     &   475   &   1.189e-008      &   9.187e-007            \\
             &        & L-BFGS           &   208   &   2.171e-010      &   8.924e-007            \\
             &        & CG-Descent 5.3   &   475   &   1.184e-008      &   9.078e-007            \\
\hline
denschna     &  2     & mBFGS            &   7    &   1.468e-014      &   3.198e-007            \\ 
             &        & arcBFGS          &   6    &   7.030e-012      &   6.598e-006            \\
             &        & L-CG-Descent     &   9    &   3.167e-016      &   3.527e-008            \\
             &        & L-BFGS           &   9    &   3.167e-016      &   3.527e-008            \\
             &        & CG-Descent 5.3   &   9    &   7.355e-016      &   4.825e-008            \\
\hline
denschnb     &  2     & mBFGS            &   7    &   6.048e-014      &   4.252e-007            \\ 
             &        & arcBFGS          &   6    &   2.675e-014      &   6.598e-006            \\
             &        & L-CG-Descent     &   7    &   3.641e-017      &   1.034e-008            \\
             &        & L-BFGS           &   7    &   3.641e-017      &   1.034e-008            \\
             &        & CG-Descent 5.3   &   8    &   4.702e-014      &   4.131e-007            \\
\hline
denschnc     &  2     & mBFGS            &   8    &  {\bf 1.119e-021} &   1.731e-010            \\ 
             &        & arcBFGS          &   14   &  {\bf 1.774e-013} &   6.791e-007            \\
             &        & L-CG-Descent     &   12   &  {\bf 3.253e-019} &   3.276e-009            \\
             &        & L-BFGS           &   12   &  {\bf 3.253e-019} &   3.276e-009            \\
             &        & CG-Descent 5.3   &   12   &  1.834e-001       &   4.143e-007            \\
\hline
denschnd     &  3     & mBFGS            &   38   &  2.461e-009       &   3.146e-007            \\ 
             &        & arcBFGS          &   76   &  4.967e-008       &   6.323e-006            \\
             &        & L-CG-Descent     &   47   &  4.331e-010       &   8.483e-007            \\
             &        & L-BFGS           &   47   &  4.331e-010       &   8.483e-007            \\
             &        & CG-Descent 5.3   &   45   &  8.800e-009       &   6.115e-007            \\
\hline
denschnf     &  2     & mBFGS            &   10   &   4.325e-018      &    3.027e-008           \\ 
             &        & arcBFGS          &   8   &   2.088e-017      &    8.453e-007           \\
             &        & L-CG-Descent     &   8    &   2.126e-015      &    6.455e-007           \\
             &        & L-BFGS           &   8    &   2.126e-015      &    6.455e-007           \\
             &        & CG-Descent 5.3   &   11   &   1.104e-017      &    6.614e-008           \\
\hline
djtl         &  2     & mBFGS            &   79   &   -8.952e+003      &    2.265e-002           \\
             &        & arcBFGS          &   86   &   -8.952e+003      &    1.116e-003           \\
             &        & L-CG-Descent     &   82   &   -8.952e+003      &    8.865e-009           \\
             &        & L-BFGS           &   82   &   -8.952e+003      &    8.865e-009           \\
             &        & CG-Descent 5.3   &   93   &   -8.952e+003      &    3.521e-007           \\
\hline
engval2      &  3     & mBFGS            &   28   &  1.999e-018       &   9.405e-008            \\ 
             &        & arcBFGS          &   31   &  3.999e-014       &   3.155e-007            \\
             &        & L-CG-Descent     &   26   &  1.034e-016       &   8.236e-007            \\
             &        & L-BFGS           &   26   &  1.034e-016       &   8.236e-007            \\
             &        & CG-Descent 5.3   &   76   &  3.185e-014       &   5.682e-007            \\
\hline
expfit       &  2     & mBFGS            &   12   &   2.405e-001      &    2.916e-009           \\ 
             &        & arcBFGS          &   11   &   2.405e-001      &    6.788e-008           \\
     	     &        & L-CG-Descent     &   13   &   2.405e-001      &    4.208e-007           \\
             &        & L-BFGS           &   13   &   2.405e-001      &    4.208e-007           \\
             &        & CG-Descent 5.3   &   15   &   2.405e-001      &    1.758e-007           \\
\hline
growthls     &  3     & mBFGS            &   1    &   3.542e+003       &    0.000e-999           \\ 
             &        & arcBFGS          &   1    &   3.542e+003       &    6.788e-008           \\
     	     &        & L-CG-Descent     &   143  &   {\bf 1.004e+000} &    3.317e-007           \\
             &        & L-BFGS           &   143  &   {\bf 1.004e+000} &    3.317e-007           \\
             &        & CG-Descent 5.3   &   441  &   {\bf 1.004e+000} &    1.835e-007           \\
\hline
hairy        &  2     & mBFGS            &   19   &  2.000e+001       &   6.143e-008            \\ 
             &        & arcBFGS          &   52   &  2.000e+001       &   1.958e-009            \\
             &        & L-CG-Descent     &   36   &  2.000e+001       &   7.961e-011            \\
             &        & L-BFGS           &   36   &  2.000e+001       &   7.961e-011            \\
             &        & CG-Descent 5.3   &   14   &  2.000e+001       &   1.044e-007            \\
\hline
hatfldd      &  3     & mBFGS            &   24   &  6.615e-008       &   1.107e-007            \\ 
             &        & arcBFGS          &   25   &  6.616e-008       &   5.515e-006            \\
	         &        & L-CG-Descent     &   20   &  2.547e-007       &   1.936e-007            \\
	         &        & L-BFGS           &   20   &  2.547e-007       &   1.936e-007            \\
             &        & CG-Descent 5.3   &   40   &  6.617e-008       &   1.934e-007            \\
\hline
hatflde      &  3     & mBFGS            &   30   &  {\bf 4.434e-007} &   6.576e-007            \\ 
             &        & arcBFGS          &   32   &  {\bf 4.434e-007} &   3.404e-006            \\
	         &        & L-CG-Descent     &   30   &  2.000e+001       &   5.012e-007            \\
	         &        & L-BFGS           &   30   &  2.000e+001       &   5.012e-007            \\
             &        & CG-Descent 5.3   &   53   &  2.000e+001       &   5.012e-007            \\
\hline
heart6ls     &  6     & mBFGS            &   2266  &  2.865e-023       &   6.934e-009        \\ 
             &        & arcBFGS                &   848  &  3.815e-016       &   8.207e-006        \\
             &        & L-CG-Descent       &   684   &  2.646e-010       &   5.562e-007        \\
             &        & L-BFGS                  &   684   &  2.646e-010       &   5.562e-007        \\
             &        & CG-Descent 5.3     &   2570  &  1.305e-010       &   2.421e-007        \\
\hline
helix        &  3     & mBFGS            &   22   &   5.489e-017      &    1.349e-007           \\ 
             &        & arcBFGS          &   28   &   9.428e-017      &    1.207e-007           \\
             &        & L-CG-Descent     &   23   &   1.604e-015      &    3.135e-007           \\
             &        & L-BFGS           &   23   &   1.604e-015      &    3.135e-007           \\
             &        & CG-Descent 5.3   &   44   &   2.427e-013      &    6.444e-007           \\
\hline 
himmelbb     &  2     & mBFGS            &   1    &   9.665e-014      &    8.167e-008           \\ 
             &        & arcBFGS          &   7    &   1.774e-014      &    4.570e-006           \\
             &        & L-CG-Descent     &   10   &   9.294e-013      &    2.375e-007           \\
             &        & L-BFGS           &   10   &   9.294e-013      &    2.375e-007           \\
             &        & CG-Descent 5.3   &   11   &   1.584e-013      &    1.084e-008           \\
\hline
himmelbg     &  2     & mBFGS            &   7    &   1.070e-013      &    9.071e-007           \\ 
             &        & arcBFGS          &   6    &   7.153e-015      &    2.463e-007           \\
             &        & L-CG-Descent     &   8    &   9.294e-013      &    2.375e-007           \\
             &        & L-BFGS           &   8    &   9.294e-013      &    2.375e-007           \\
             &        & CG-Descent 5.3   &   10   &   1.584e-013      &    1.084e-008           \\
\hline
himmelbh     &  2     & mBFGS            &   7    &    -1.000e+000     &    5.026e-007           \\ 
             &        & arcBFGS          &   5    &    -1.000e+000     &    1.029e-007           \\
             &        & L-CG-Descent     &   7    &    -1.000e+000     &    2.892e-011           \\
             &        & L-BFGS           &   7    &    -1.000e+000     &    2.892e-011           \\
             &        & CG-Descent 5.3   &   7    &    -1.000e+000     &    1.381e-007           \\
\hline
humps        &  2     & mBFGS            &   104  &    3.280e-016     &    6.351e-009           \\ 
             &        & arcBFGS          &   67   &    1.749e-010     &    5.918e-006           \\
             &        & L-CG-Descent     &   53   &    3.682e-012     &    8.552e-007           \\
             &        & L-BFGS           &   53   &    3.682e-012     &    8.552e-007           \\
             &        & CG-Descent 5.3   &   48   &    3.916e-012     &    8.774e-007           \\
\hline
jensmp       &  2     & mBFGS            &   1    &    2.020e+003       &    0.000e-999           \\ 
             &        & arcBFGS          &   1   &    {\bf 1.244e+002} &    7.035e-008           \\
             &        & L-CG-Descent     &   15   &    {\bf 1.244e+002} &    5.302e-010           \\
             &        & L-BFGS           &   15   &    {\bf 1.244e+002} &    5.302e-010           \\
             &        & CG-Descent 5.3   &   13   &    {\bf 1.244e+002} &    4.206e-009           \\
\hline
kowosb       &  4     & mBFGS            &   28   &    3.075e-004     &    1.367e-007           \\ 
             &        & arcBFGS          &   25   &    3.075e-004     &    3.686e-006           \\
             &        & L-CG-Descent     &   17   &    3.078e-004     &    3.704e-007           \\
             &        & L-BFGS           &   17   &    3.078e-004     &    3.704e-007           \\
             &        & CG-Descent 5.3   &   66   &    3.078e-004     &    8.818e-007           \\
\hline
loghairy     &  2     & mBFGS            &   74   &    1.823e-001     &    5.904e-007           \\ 
             &        & arcBFGS          &   78  &    1.823e-001     &  3.686e-006           \\
             &        & L-CG-Descent     &   27   &    1.823e-001     &    1.762e-007           \\
             &        & L-BFGS           &   27   &    1.823e-001     &    1.762e-007           \\
             &        & CG-Descent 5.3   &   46   &    1.823e-001     &    7.562e-008           \\
\hline
mancino      &  100   & mBFGS            &   37   &     1.548e-020    &     1.414e-007          \\ 
             &        & arcBFGS          &   122  &     7.499e-018    &     7.663e-006          \\
             &        & L-CG-Descent     &   11   &     9.245e-021    &     7.239e-008          \\
             &        & L-BFGS           &   9    &     3.048e-021    &     1.576e-007          \\
             &        & CG-Descent 5.3   &   11   &     9.245e-021    &     7.239e-008          \\
\hline
maratosb     &  2     & mBFGS            &   3    &   -1.000e+000      &   5.142e-008            \\ 
             &        & arcBFGS          &   2    &   -1.000e+000      &   1.767e-008            \\
             &        & L-CG-Descent     &   1145 &   -1.000e+000      &   3.216e-007            \\
             &        & L-BFGS           &   1145 &   -1.000e+000      &   3.216e-007            \\
             &        & CG-Descent 5.3   &   946  &   -1.000e+000      &   3.230e-009            \\
\hline
mexhat       &  2     & mBFGS            &   5    &    -4.010e-002     &     1.426e-012          \\ 
             &        & arcBFGS          &  1    &    -4.010e-002     &     6.706e-006          \\
             &        & L-CG-Descent     &   20   &    -4.001e-002     &     4.934e-009          \\
             &        & L-BFGS           &   20   &    -4.001e-002     &     4.934e-009          \\
             &        & CG-Descent 5.3   &   27   &    -4.001e-002     &     3.014e-007          \\
\hline
osborneb     &  11    & mBFGS            &   53   &    4.014e-002     &     2.480e-007          \\ 
             &        & arcBFGS          &   49   &    4.014e-002     &     6.522e-006          \\
	         &        & L-CG-Descent     &   62   &    4.014e-002     &     4.427e-007          \\
	         &        & L-BFGS           &   62   &    4.014e-002     &     4.427e-007          \\
             &        & CG-Descent 5.3   &  214   &    4.014e-002     &     7.485e-007          \\
\hline
palmer1c     &  8     & mBFGS            &   32   &    9.760e-002     &     3.935e-007          \\ 
             &        & arcBFGS          &   33   &    9.760e-002     &     3.119e-004          \\
             &        & L-CG-Descent     &   11   &    9.761e-002     &     1.254e-009          \\
             &        & L-BFGS           &   11   &    9.761e-002     &     1.254e-009          \\
             &        & CG-Descent 5.3   & 126827 &    9.761e-002     &     9.545e-007          \\
\hline
palmer2c     &  8     & mBFGS            &   112  &   1.442e-002      &     8.296e-007          \\ 
             &        & arcBFGS          &   38   &   1.442e-002      &     3.648e-006          \\
             &        & L-CG-Descent     &   11   &   1.437e-002      &     1.257e-008          \\
             &        & L-BFGS           &   11   &   1.437e-002      &     1.257e-008          \\
             &        & CG-Descent 5.3   &  21362 &   1.437e-002      &     5.761e-007          \\
\hline
palmer3c     &  8     & mBFGS            &   47   &   1.954e-002      &     2.050e-008          \\ 
             &        & arcBFGS          &   38   &   1.954e-002      &     3.648e-006          \\
             &        & L-CG-Descent     &   11   &   1.954e-002      &     1.754e-010          \\
             &        & L-BFGS           &   11   &   1.954e-002      &     1.754e-010          \\
             &        & CG-Descent 5.3   &   5536 &   1.954e-002      &     9.753e-007          \\
\hline
palmer4c     &  8     & mBFGS            &   78   &   5.031e-002      &     2.235e-007          \\ 
             &        & arcBFGS          &   36   &   5.031e-002      &     9.770e-006          \\
             &        & L-CG-Descent     &   11   &   5.031e-002      &     3.928e-009          \\
             &        & L-BFGS           &   11   &   5.031e-002      &     3.928e-009          \\
             &        & CG-Descent 5.3   &  44211 &   5.031e-002      &     9.657e-007          \\
\hline
palmer5c     &  6     & mBFGS            &   13   &    2.128e+000     &      4.810e-009         \\ 
             &        & arcBFGS          &   11   &    2.128e+000     &      1.713e-007         \\
             &        & L-CG-Descent     &   6    &    2.128e+000     &      3.749e-012         \\
             &        & L-BFGS           &   6    &    2.128e+000     &      3.749e-012         \\
             &        & CG-Descent 5.3   &   6    &    2.128e+000     &      2.629e-009         \\
\hline
palmer6c     &  8     & mBFGS            &   56   &   1.639e-002      &     6.900e-007          \\ 
             &        & arcBFGS          &   60   &   1.639e-002      &     6.727e-006          \\
             &        & L-CG-Descent     &   11   &   1.639e-002      &     5.520e-009          \\
             &        & L-BFGS           &   11   &   1.639e-002      &     5.520e-009          \\
             &        & CG-Descent 5.3   &  14174 &   1.639e-002      &     7.738e-007          \\
\hline
palmer7c     &  8     & mBFGS            &   41   &   6.020e-001      &     5.201e-007          \\ 
             &        & arcBFGS          &   44   &   6.019e-001      &     1.343e-006          \\
             &        & L-CG-Descent     &   11   &   6.020e-001      &     7.132e-009          \\
             &        & L-BFGS           &   11   &   6.020e-001      &     7.132e-009          \\
             &        & CG-Descent 5.3   & 65294  &   6.020e-001      &     9.957e-007          \\
\hline
palmer8c     &  8     & mBFGS            &   48   &   1.598e-001      &     1.099e-009          \\ 
             &        & arcBFGS          &   50   &   1.598e-001      &     1.565e-006          \\
             &        & L-CG-Descent     &   11   &   1.598e-001      &     2.376e-009          \\
             &        & L-BFGS           &   11   &   1.598e-001      &     2.376e-009          \\
             &        & CG-Descent 5.3   &   8935 &   1.598e-001      &     9.394e-007          \\
\hline
rosenbr      &  2     & mBFGS            &   32   &   1.383e-016      &     4.603e-007          \\ 
             &        & arcBFGS          &   23   &   2.171e-015      &     9.558e-006          \\
             &        & L-CG-Descent     &   34   &   4.691e-018      &     7.167e-008          \\
             &        & L-BFGS           &   34   &   4.691e-018      &     7.167e-008          \\
             &        & CG-Descent 5.3   &   37   &   1.004e-014      &     1.894e-007          \\
\hline
sineval      &  2     & mBFGS            &   69   &   1.910e-019      &     1.168e-008          \\ 
             &        & arcBFGS          &   65   &   2.050e-015      &     1.175e-006          \\
             &        & L-CG-Descent     &   60   &   1.556e-023      &     1.817e-011          \\
             &        & L-BFGS           &   60   &   1.556e-023      &     1.817e-011          \\
             &        & CG-Descent 5.3   &   62   &   1.023e-012      &     5.575e-007          \\
\hline
sisser       &  2     & mBFGS            &   19   &    3.860e-010     &     4.587e-007          \\ 
             &        & arcBFGS          &   16   &    3.015e-008     &     9.980e-006          \\
             &        & L-CG-Descent     &   6    &    6.830e-012     &     2.220e-008          \\
             &        & L-BFGS           &   6    &    6.830e-012     &     2.220e-008          \\
             &        & CG-Descent 5.3   &   6    &    3.026e-014     &     3.663e-010          \\
\hline
tointqor     &  50    & mBFGS            &   39   &   1.176e+003      &     4.033e-007          \\ 
             &        & arcBFGS          &   35   &   1.175e+003      &     4.874e-006          \\
             &        & L-CG-Descent     &   29   &   1.175e+003      &     4.467e-007          \\
             &        & L-BFGS           &   28   &   1.175e+003      &     7.482e-007          \\
             &        & CG-Descent 5.3   &   29   &   1.175e+003      &     4.464e-007          \\
\hline
vardim       &  200   & mBFGS            &   22   &   1.237e-021      &     1.376e-009          \\ 
             &        & arcBFGS          &   21   &   5.428e-008      &     9.980e-006          \\
             &        & L-CG-Descent     &   10   &   4.168e-019      &     2.582e-007          \\
             &        & L-BFGS           &   7    &   5.890e-025      &     3.070e-010          \\
             &        & CG-Descent 5.3   &   10   &   4.168e-019      &     2.582e-007          \\
\hline
watson       &  12    & mBFGS            &   61   &   1.130e-008      &     3.081e-007          \\ 
             &        & arcBFGS          &   46   &   2.584e-007      &     8.722e-006          \\
             &        & L-CG-Descent     &   49   &   1.592e-007      &     8.026e-007          \\
             &        & L-BFGS           &   48   &   9.340e-008      &     1.319e-007          \\
             &        & CG-Descent 5.3   &   726  &   1.139e-007      &     8.115e-007          \\
\hline
yfitu        &  2     & mBFGS            &   73   &   6.670e-013      &     1.938e-007          \\ 
             &        & arcBFGS          &   63  &   6.704e-013      &     4.874e-006          \\
             &        & L-CG-Descent     &   75   &   8.074e-010      &     3.910e-007          \\
             &        & L-BFGS           &   75   &   8.074e-010      &     3.910e-007          \\
             &        & CG-Descent 5.3   &   147  &   2.969e-011      &     5.681e-007          \\
\hline
\end{longtable}
\end{center}
\normalsize

We summarize the comparison of the test results as follows:
\begin{itemize}
\item[1.] Generally, all algorithms find the same optimal solutions for most problems. But for $7$ problems ({\tt arglina}, {\tt bard}, {\tt biggs6}, {\tt denschnc}, {\tt growthls}, {\tt hatflde}, and {\tt jensmp}), these algorithms find different local minimums. Among them, {\tt arcBFGS} finds the best local minimums $6$ times; {\tt mBFGS}, {\tt L-CG-Descent}, and {\tt L-BFGS} find the best local minimums $4$ times; finally {\tt CG-Descent 5.3} finds the best local minimums $3$ times.
\item[2.] All algorithms find the optimal solutions for all problems but {\tt mBFGS} and {\tt arcBFGS} terminate early for problem {\tt djtl} (although they find the optimal solution). However, If we use (\ref{arcBFGS}) and (\ref{dirMBFGS}) instead of (\ref{E1}) and (\ref{calMBFGS}) in the calculation of the search direction $d_k$, {\tt mBFGS} and {\tt arcBFGS} converge and meet the criterion $\| g(x_k) \|< 10^{-5}$. Given the fact that using (\ref{E1}) and (\ref{calMBFGS}) does not require to solve the linear systems of equations but using (\ref{arcBFGS}) and (\ref{dirMBFGS}) does, we suggest using the implementation of (\ref{E1}) and (\ref{calMBFGS}).
\end{itemize}

The performance profiles for {\tt mBFGS}, {\tt arcBFGS}, {\tt L-CG-Descent}, {\tt L-BFGS}, and {\tt CG-Descent 5.3} are presented in Figure~\ref{fig:profile2}. Clearly, the proposed arc-search BFGS method requires the fewest iterations overall among all the algorithms considered. It is followed by the robust BFGS implementation {\tt mBFGS}, which also requires fewer iterations than the remaining algorithms on the tested problems. The performances of {\tt L-CG-Descent} and {\tt L-BFGS} are comparable, and both require fewer iterations than {\tt CG-Descent 5.3}.

\begin{figure}[ht]
\centerline{\epsfig{file=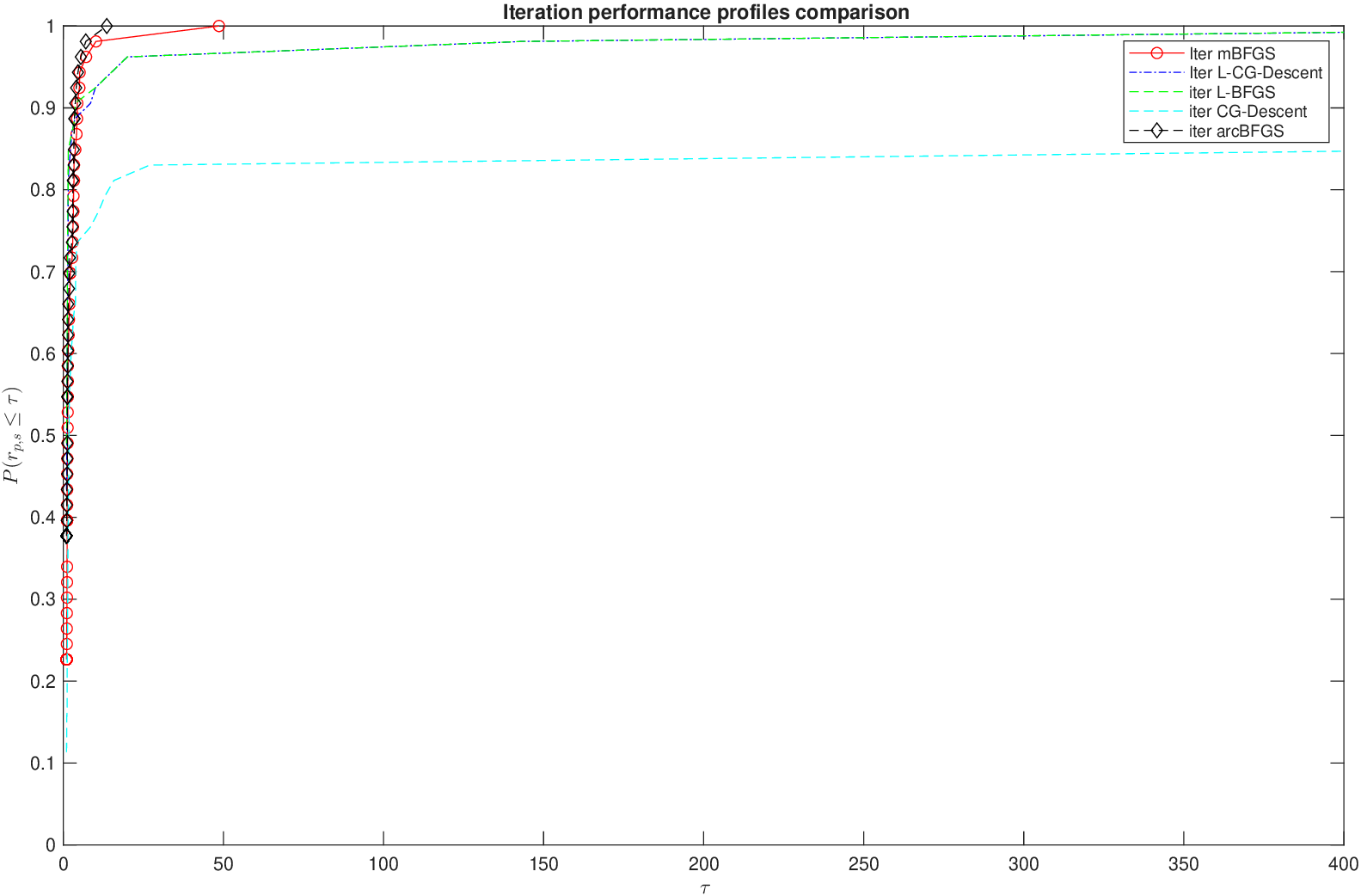,height=6cm,width=10cm}}
\caption{Performance profiles comparison for {\tt mBFGS}, {\tt arcBFGS}, 
{\tt L-CG-Descent}, {\tt L-BFGS}, {\tt CG-Descent 5.3}, and {\tt arcBFGS}.}
\label{fig:profile2}
\end{figure}

\section{Conclusions}\label{conclusions}
We have proposed an arc-search BFGS algorithm that is globally convergent to a local optimum and achieves a superlinear convergence rate for both convex and nonconvex optimization problems under mild assumptions. We also showed that the computational cost per iteration of the proposed arc-search BFGS algorithm is nearly the same as that of the classical BFGS method. Furthermore, we presented numerical experiments comparing the performance of the proposed arc-search BFGS algorithm with several established and state-of-the-art methods, including the classical BFGS method, limited-memory BFGS, descent and conjugate gradient methods, and limited-memory descent and conjugate gradient methods. The numerical results demonstrate that the proposed arc-search BFGS algorithm is both efficient and effective in practice.

\section{Acknowledgments}
The author thanks Professor Robert Vanderbei of Princeton University for providing access to his website~\cite{web}, which includes the CUTEst test problems used in this paper. The author is also grateful to Professor Teresa Monteiro of the Universidade do Minho for assistance in converting CUTEst AMPL files to .nl files, and to Professor António Ismael Freitas Vaz, also of the Universidade do Minho, for providing the {\tt matampl} function, which enables the use of MATLAB to test CUTEst nonlinear optimization problems.

\section{Conflict of interest}

The author declares no conflict of interest.

\section{Code availability}

The MATLAB code and test problems used in this study are available from the author upon reasonable request.



\end{document}